\documentclass[a4paper,12pt]{article}
\usepackage[margin=1in]{geometry}
\usepackage[small,labelfont=bf]{caption}
\usepackage{amsmath,amsfonts,amssymb,graphicx,bbm,stmaryrd}
\usepackage{hyperref}
\usepackage{color}

\newtheorem{theorem}{Theorem}
\newtheorem{lemma}[theorem]{Lemma}
\newtheorem{definition}[theorem]{Definition}
\newtheorem{proposition}[theorem]{Proposition}

\newtheorem{conjecture}{Conjecture}
\numberwithin{equation}{section}

\newenvironment{proof}[1][\relax]%
 {\paragraph{Proof\ifx#1\relax:\else~of #1:\fi}}%
 {~\hfill$\square$\par\bigskip}

\newcommand{\Z}{\mathbb Z}

\title{Divergence of the correlation length for critical planar FK percolation with $1\le q\le4$ via parafermionic observables}

\author{H. Duminil-Copin}
\begin{document}

\maketitle

\begin{abstract}
Parafermionic observables were introduced by Smirnov for planar FK percolation in order to study the critical phase $(p,q)=(p_c(q),q)$. This article gathers several known properties of these observables.  Some of these properties are used to prove the divergence of the correlation length when approaching the critical point for FK percolation when $1\le q\le 4$. A crucial step is to consider FK percolation on the universal cover of the punctured plane. We also mention several conjectures on FK percolation with arbitrary cluster-weight $q>0$.
\end{abstract}

\section{Introduction}

\paragraph{Definition of the model} Since its introduction by Fortuin and Kasteleyn \cite{FK72}, the Fortuin-Kasteleyn (FK) percolation has become an important tool in the 
study of phase transitions. The spin correlations of Potts models are 
rephrased as cluster connectivity properties of their FK 
representations via the Edwards and Sokal coupling. This allows for the use of geometric techniques, thus 
leading to several important applications. For example, Swendsen and Wang utilized the model in proposing an
algorithm for the time-evolution of Potts models \cite{SW87}. Another example is provided by the recent classification of planar Gibbs states \cite{CDIV12}. See \cite{Bax89,Gri06} for  more applications. 

The FK percolation on a finite subgraph of the square lattice $\Z^2$ is a probability measure on edge configurations (each edge is declared open or closed) such that the probability of a configuration is proportional to $p^{\#\,\text{ open edges}}(1-p)^{\#\,\text{ closed edges}}q^{\#\,\text{clusters}}$, where clusters are maximal graphs connected by open edges.

A dual configuration can be defined on the dual graph $(\Z^2)^*$ by declaring every dual edge open if the corresponding edge of the primal graph is closed, and vice-versa; see Fig.~\ref{fig:1}. The dual configuration is then distributed as a FK percolation with parameters $(p^*,q^*)$ given by $q^*=q$ and $\frac{p^* p}{(1-p^*)(1-p)}=q$. It was shown in \cite{BD12} that the critical point for the FK percolation with $q\ge 1$ is given by the unique solution of $p^*(p,q)=p$, i.e. $p_c(q)=\sqrt q/(1+\sqrt q)$ (the case $q\ge 4$ was resolved much earlier in \cite{HKW78}).

Critical FK percolation exhibits a very rich behavior depending on the value of cluster-weight $q$. Exact computations in specific geometries (see e.g. \cite{Bax71,Bax73,Bax89} or the review \cite{Wu82} and references therein)
provide very precise results on the behavior of these models 
at and near criticality. It is therefore fair to say that the FK percolation is one of the most understood model of planar statistical physics.

The goal of this article is to provide an alternative approach to questions on the critical FK percolation, based on parafermionic observables rather than exact computations.

\paragraph{Parafermionic observable} Recently, an observable of the loop model, called the {\em fermonic observable}, has been introduced in the $q=2$ case \cite{Smi10,CR06} (for such value of $q$, the model
can be coupled with the Ising model via the Edwards-Sokal coupling \cite{ES88}). This observable was proved to be preholomorphic (meaning that it is a relevant discretization of a holomorphic function) and to converge in the scaling limit to a conformally covariant object. The article \cite{Smi10} also mentioned the possible generalization of this observable to other values of $q$ in $(0,4]$. In this case, the observable is a (anti)-holomorphic parafermion of fractional spin $\sigma\in[0,1]$, 
given by certain vertex operators. Similarly to the $q=2$ case, the observable is believed to converge to a conformally covariant object and to provide a deep understanding of the critical regime. 

In this article, we recall the definition of the parafermionic observable for general $q\in(0,4)$ and present several of its properties. We also introduce a slightly different observable in the $q=4$ case. The observable is shown to satisfy local relations (Proposition~\ref{relation_around_a_vertex}) that can be understood as discretizations of the Cauchy-Riemann equations when the model is critical (this proof is an easy extension of a result in \cite{Smi10}). Unfortunately, local relations provide us with half of the discrete Cauchy-Riemann equations only, and the observable is not fully preholomorphic, but rather a divergence-free differential form, in the sense that its discrete integrals along contours vanish. As mentioned above, for $q=2$, further information can be extracted from local relations and the observable satisfies a strong notion of preholomorphicity. In this case, the observable can be used to understand many properties on the model, including conformal invariance of the observable  \cite{CS09,Smi10} and loops \cite{CDHKS12,HK11}, correlations \cite{CHI12,CI12,Hon10,HS10} and crossing probabilities \cite{BDH12,CDH12,DHN10}. It can also be extended away from criticality \cite{BDC11a,DGP12}. We do not discuss the special feature of the $q=2$ case and we refer to the extensive literature for further information. 

Even though the observable is not fully preholomorphic for generic $q\in(0,4]$, it still satisfies a special property at $p_c(q)$. This property can be used to derive information on the model, and we would like to discuss two applications, one rigorous, and one conjectural.

\paragraph{First application of parafermionic observables} Using the parafermionic observable, we are able to prove that the correlation length diverges when $1\le q\le 4$. 
\begin{theorem}\label{thm:correlation length}
Fix $1\le q\le 4$, the correlation length $\xi(p)$ tends to infinity when $p\nearrow p_c(q),$
where $$\xi(p)^{-1}~=~-\inf_{n>0}\frac1n\log \phi^0_{\mathbb Z^2,p,q}(0\longleftrightarrow (n,0)).$$
\end{theorem}

In the statement, $\phi^0_{\mathbb Z^2,p,q}$ is the infinite-volume FK percolation measure with free boundary conditions, and $a\longleftrightarrow b$ if there exists a path of open edges from vertex $a$ to vertex $b$. In fact, when $1\le q\le 3$, a stronger result can be proved: \begin{theorem}\label{thm:susceptibility}
When $1\le q\le 3$, the susceptibility $\displaystyle \sum_{x\in \Z^2}\phi^0_{\mathbb Z^2,p_c,q}(0\longleftrightarrow x)$ equals $\infty.$
\end{theorem}

The reason for working with FK percolation with cluster-weight $q\ge 1$ and not arbitrary weight $q>0$ will become clear later. Some techniques involved in the proof invoke the FKG inequality \cite[Theorem 3.8]{Gri06}, a tool which is not present for $q<1$. Let us mention that the parafermionic observables are still available for $q<1$, and corresponding predictions can be made. 

Theorem~\ref{thm:correlation length} has the following interpretation. Ehrenfest 
classified phase transitions based on the behavior of the thermodynamical free energy viewed as a function of other thermodynamical quantities. He defined the order of the phase transition as the lowest derivative of the free energy  which is discontinuous at the phase transition. For instance, the free energy is continuous yet non-differentiable when the transition is of first order. In FK percolation, the phase transition is believed to be of second order if and only if the correlation length diverges when approaching criticality. It is of first order otherwise. As a consequence, Theorem~\ref{thm:correlation length} strongly suggests that the phase transition is second order for $q<4$. 
This result is optimal in some sense, since FK percolation with $q>4$ undergoes a first order phase transition. In this case, Theorem~\ref{thm:correlation length} is no longer valid.

\paragraph{Second application parafermionic observables} This discussion is mostly due to Smirnov and Schramm. 
In 1999, Schramm \cite{Sch00} noticed that interfaces in planar models satisfy the domain Markov property, which, together with the assumption of conformal invariance, determines a one-parameter family of continuous random non-self-crossing curves, now called {\em Schramm-Loewner Evolution} ($\mathsf{SLE}$ for short). For $\kappa>0$, the $\mathsf{SLE}$($\kappa$) is the random Loewner Evolution with driving process $\sqrt \kappa B_t$, where $(B_t)$ is a standard Brownian motion. Since its introduction, $\mathsf{SLE}$ has been central in planar statistical physics and Conformal Field Theory, in particular because it provides a mathematical framework for the study of these interfaces. We refer to \cite{LSW01b,LSW01a,RS05} for a description of the fundamental fractal properties of $\mathsf{SLE}$s and to \cite{KN04} for an introduction intended for physicists. 

Proving convergence of the discrete interfaces of a certain model to $\mathsf{SLE}$ is crucial since the path properties of $\mathsf{SLE}$s are related to fractal properties of the critical models, and therefore to critical exponents (see \cite{Sch07,Smi06} for a collection of problems). The standard path to prove convergence starts by exhibiting a discrete observable converging to a conformally covariant object in the scaling limit. Holomorphic solutions to Dirichlet or Riemann boundary value problems are archetypical examples of conformally covariant objects. Therefore, it is natural to expect that discrete observables which are conformally covariant in the scaling limit are naturally preholomorphic functions which are solutions of discrete boundary problems. Finding such observables have been at the heart of planar statistical physics this last decade. 
Unfortunately, except in exceptional cases (dimers and uniform spanning trees, see \cite{LSW04a,Ken00,Ken01}, as well as Ising and FK percolation with cluster-weight $q=2$, see \cite{CS09,Smi10}), no fully preholomorphic observables have been found at the discrete level, and the best available candidates only satisfy part of discrete Cauchy-Riemann equations. 
The parafermionic observable is a typical such example, which is conjectured to converge to a conformally covariant observable. Even though a rigorous proof of this convergence remains open, one can use the parafermionic observable to predict towards which $\mathsf{SLE}(\kappa)$ the interfaces of the FK percolation with cluster-weight $q$ converges. Furthermore, the observable could  a priori be used to prove convergence to $\mathsf{SLE}$, and we intend to explain the general methodology in this paper.

\paragraph{Other models} The parafermionic observable was also introduced in the context of the high-temperature expansion of the Ising model (or $O(1)$-model) to prove convergence of the Ising interfaces towards $\mathsf{SLE}$(3) \cite{CS09}. It was later generalized to the case of loop $O(n)$-models on the hexagonal lattice. It can be proved that the discrete contour integrals of the observable vanish at $x=\sqrt{2+\sqrt{2-n}}$, where $x$ is the edge-weight of the $O(n)$-model \cite{Smi10b}. Unfortunately, the mathematical understanding of these models is fairly restricted and applications of the observables for $n\ne 1$ are restricted to few examples:
\begin{itemize}
\item arguments closely related to those exposed in this paper allow one to prove that the susceptibility is infinite at $x=1/\sqrt{2+\sqrt{2-n}}$, thus showing that a phase transition occurs, and that $1/\sqrt{2+\sqrt{2-n}}$ is an upper bound for the critical parameter $x_c$ (Nienhuis conjectured that $x_c=1/\sqrt{2+\sqrt{2-n}}$ in \cite{Nie82,Nie84}). 
\item In the $n=0$ case (corresponding to the self-avoiding walk model), the connective constant of the hexagonal lattice can be shown to be equal to $\sqrt {2+\sqrt 2}$ \cite{DS12}.
\item Let us mention without details that there are other applications \cite{BBDGG12,BGG12b,BGG12a,EGGL12}.
\end{itemize}
Later, such observables have been found in a variety of lattice models with specific weights (for instance $O(n)$-models on the square lattice and $Z_N$ models, see \cite{CI09,CR06}). 
Interestingly, weights for which discrete contour integrals of these (non-degenerate) observables vanish always correspond to weights for which Yang-Baxter equations hold. In \cite{Car09}, Cardy asks whether a direct link can be found between these two notions, and this question is probably crucial for the future development of the theory.

\paragraph{Organization of the paper} In the next section, the loop representation of the FK percolation is introduced, and the parafermionic observable is defined. Section~\ref{properties} is a discussion on the observable. We list some of its properties, and we explain how the observable is related to $\mathsf{SLE}$ theory. Section~\ref{sec:1<q<4} contains the proofs of Theorems~\ref{thm:correlation length} and \ref{thm:susceptibility}. We also introduce a parafermionic observable in the degenerated case $q=4$. Section~\ref{open questions} gathers open questions.

\paragraph{Notations}
We consider the square lattice $\mathbb Z^2$ with vertex set $\mathbb Z^2$ and edges between nearest neighbors. The {\em dual lattice} $(\mathbb Z^2)^*=(\frac12,\frac12)+\mathbb Z^2$ is given by sites corresponding to every face of $\mathbb Z^2$, and edges linking nearest neighbors. The {\em medial lattice} $(\mathbb Z^2)^\diamond$ is defined as follows: its vertices are the mid-edges of $\mathbb Z^2$, and its edges connect nearest neighbors. This lattice is a rotated and rescaled version of the square lattice. We orient every medial edge counterclockwise around faces corresponding to sites of $\mathbb Z^2$. For a graph $G$, $G^*$ and $G^\diamond$ denote the dual and the medial graphs of $G$. The boundary of a graph $G$ will be denoted by $\partial G$ (it will be clear from the context whether we consider site or edge boundary).

\paragraph{Acknowledgments}
% The field of FK percolation would not be what it is today without the extraordinary body of works on exact solutions to planar statistical physics, in particular the Wu and his coauthors for their exceptional contribution to this field. 
The author would like to thank Stanislav Smirnov for introducing him to this beautiful subject and for many fruitful discussions. We also thank Aernout van Enter for his comments on the manuscript and for valuable discussions. The author was supported by the ANR grant BLAN06-3-134462, the ERC AG CONFRA, as well as by the Swiss
{FNS}.
\section{FK percolation}\label{sec:RCM}

In order to remain as self-contained as possible, we introduce the FK percolation precisely, in particular its different representations and its boundary conditions. The reader can consult the reference book \cite{Gri06} for more details, proofs and original references. 

\paragraph{Definition of the model on $\mathbb Z^2$}

Let $G$ be a finite subgraph of $\mathbb Z^2$. A \emph{configuration} $\omega$ on $G$ is a subgraph of $G$, composed of 
the same sites and a subset of its edges. The edges 
belonging to $\omega$ are called \emph{open}, the others \emph{closed}. Two sites 
$a$ and $b$ are said to be \emph{connected} if there is an \emph{open 
path}, \emph{i.e.} a path composed of open edges only, connecting them. Two sets $A$ and $B$ are 
\emph{connected} if there exists an open path connecting them (this event is denoted by 
$A\longleftrightarrow B$). Maximal connected components will be called 
\emph{clusters}. 

\emph{Boundary conditions} $\xi$ are given by a partition of $\partial G$. The graph obtained from the configuration 
$\omega$ by identifying (or \emph{wiring}) the edges in $\xi$ that 
belong to the same component of $\xi$ is denoted by $\omega \cup \xi$. Boundary conditions should be 
understood as an encoding of how sites are connected outside of $G$. Since the model exhibits long range dependency, boundary conditions are crucial. Let 
$o(\omega)$ (resp. $c(\omega)$) denote the number of open (resp.\ 
closed) edges of $\omega$ and $k(\omega,\xi)$ the number of connected 
components of $\omega\cup\xi$.  
The probability measure 
$\phi^{\xi}_{G,p,q}$ of the FK percolation on $G$ with parameters 
$p$ and $q$ and boundary conditions $\xi$ is defined by
\begin{equation}
  \label{probconf}
  \phi_{G,p,q}^{\xi} (\left\{\omega\right\}) :=
  \frac {p^{o(\omega)}(1-p)^{c(\omega)}q^{k(\omega,\xi)}}
  {Z_{G,p,q}^{\xi}}
\end{equation}
for every configuration $\omega$ on $G$, where $Z_{G,p,q}^{\xi}$ is a 
normalizing constant referred to as the \emph{partition function}. 

Three types of boundary conditions play a special role in the study of the
FK percolation:
\begin{enumerate}
\item The \emph{wired} boundary conditions, denoted by 
$\phi_{G,p,q}^1$, are specified by the fact that all the vertices on the 
boundary are pairwise wired (only one set in the partition). 
\item The \emph{free} boundary conditions, denoted by $\phi_{G,p,q}^0$, are 
specified by the absence of wirings between sites. 
\item The \emph{Dobrushin} boundary conditions: assume that $\partial G$ is a
self-avoiding polygon in $\mathbb Z^2$, let $a$ and $b$ be two sites
of $\partial G$. The triplet $(G,a,b)$ is called a {\em Dobrushin domain}. Orienting its boundary counterclockwise defines two oriented
boundary arcs $\partial_{ab}$ and $\partial_{ba}$; the Dobrushin boundary conditions are
defined to be free on $\partial_{ab}$ (there are no wirings between boundary
sites) and wired on $\partial_{ba}$ (all the boundary sites are pairwise
connected). These arcs are referred to as the \emph{free arc} and the
\emph{wired arc}, respectively. The measure associated to these boundary
conditions will be denoted by $\phi_{G,p,q}^{a,b}$.  Dobrushin boundary conditions are usually formulated for the spin Ising model and amount to setting plus spin boundary condition on $\partial_{ab}$ and minus spin boundary conditions on $\partial_{ba}$, thus creating an interfaces between pluses and minuses. Since we also need an interface here, we formulated similar conditions in the FK setting.\end{enumerate}

For $q\ge 1$, the FK percolation measure is positively correlated. In particular, it satisfies the FKG inequality \cite[Theorem 3.8]{Gri06}:
\begin{equation}\label{FKG}\phi_{G,p,q}^\xi(A\cap B)\ge \phi_{G,p,q}^\xi(A)\phi_{G,p,q}^\xi(B),\quad \forall A,B\text{ increasing},\quad \forall \xi \end{equation}
and the comparison between boundary conditions \cite[Lemma 4.14]{Gri06}, which is a direct consequence of the FKG inequality, 
\begin{equation}\label{comparison}\phi_{G,p,q}^\xi(A)\ge \phi_{G,p,q}^\psi(A),\quad \forall A\text{ increasing},\quad \forall \xi\ge \psi. \end{equation}
Above, $\xi\ge \psi$ if two wired vertices in $\psi$ are wired in $\xi$. For instance, free boundary conditions are the smallest possible boundary conditions, while wired are the largest. We will say that $\xi$ dominates $\psi$ and $\psi$ is dominated by $\xi$. The FK measure can be extended to the whole lattice $\mathbb Z^2$ by considering the limit of FK percolation measures with free boundary conditions on nested boxes (via comparison between boundary conditions, these measures form an increasing family of measures). We call the infinite-volume FK percolation measure $\phi^0_{\mathbb Z^2,p,q}$.
\paragraph{Dual representation}
As mentioned in the introduction, a configuration $\omega$ can be uniquely associated to a dual configuration on the dual graph $G^*$: each edge of the dual graph being open (resp. closed) if the corresponding edge of the primal graph is closed (resp. open) in $\omega$, see Fig.~\ref{fig:1}. We will often speak of dual-clusters or dual-open paths to refer to objects in this dual model. The configuration thus obtained is denoted $\omega^*$. Euler's formula together with a simple computation implies that  $\omega^*$ is distributed according to $\phi_{G^*,p^*,q^*}$ with $q^*=q$ and $\frac{pp^*}{(1-p)(1-p^*)}=q$. In particular, the unique $p$ such that $p=p^*$ is the critical point as shown in \cite{BD12} and \cite{HKW78}. In the future, $p_c=p_c(q)$ denotes the critical parameter of the FK percolation with cluster-weight $q$.

\begin{figure}[t]
\begin{center}
\includegraphics[width=4cm]{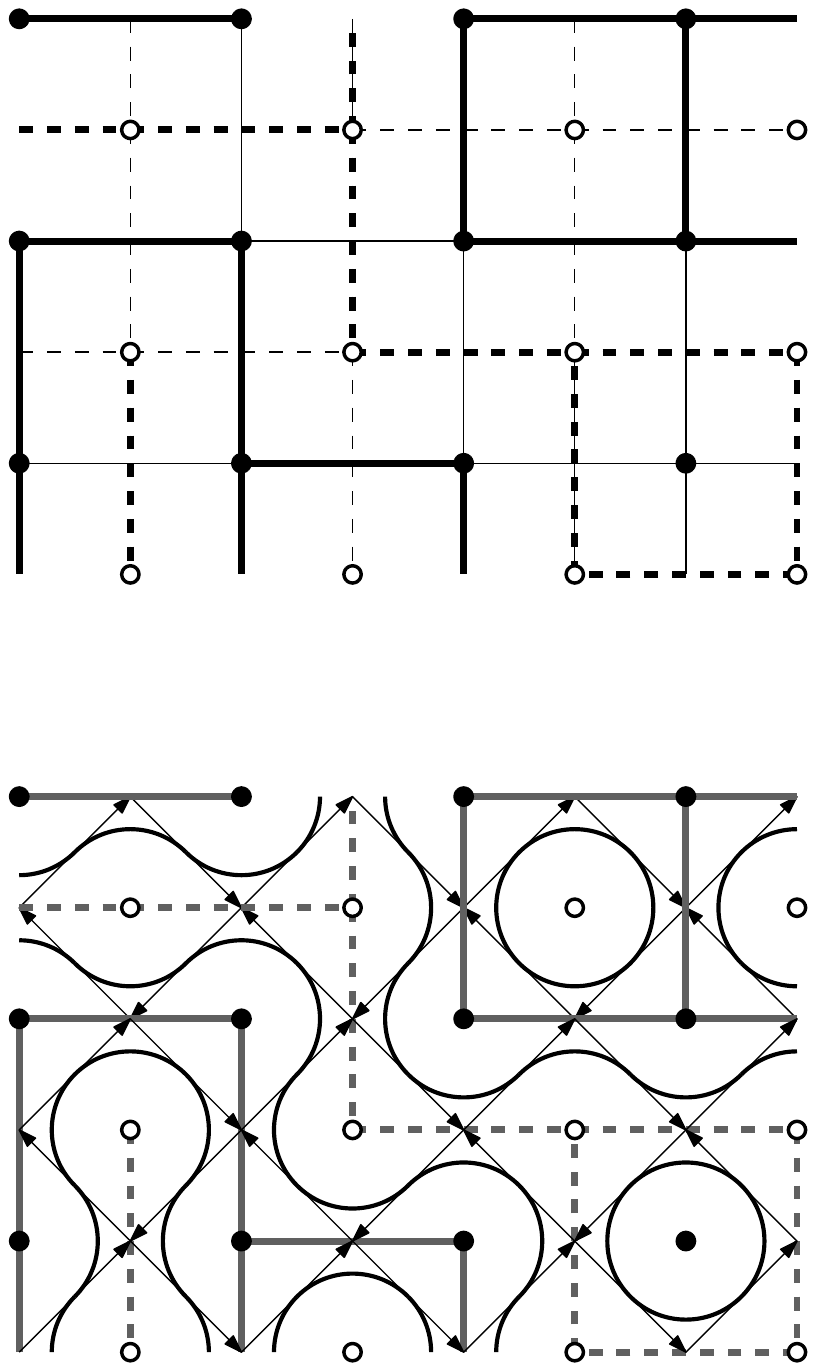}\quad\quad\quad\includegraphics[width=10cm]{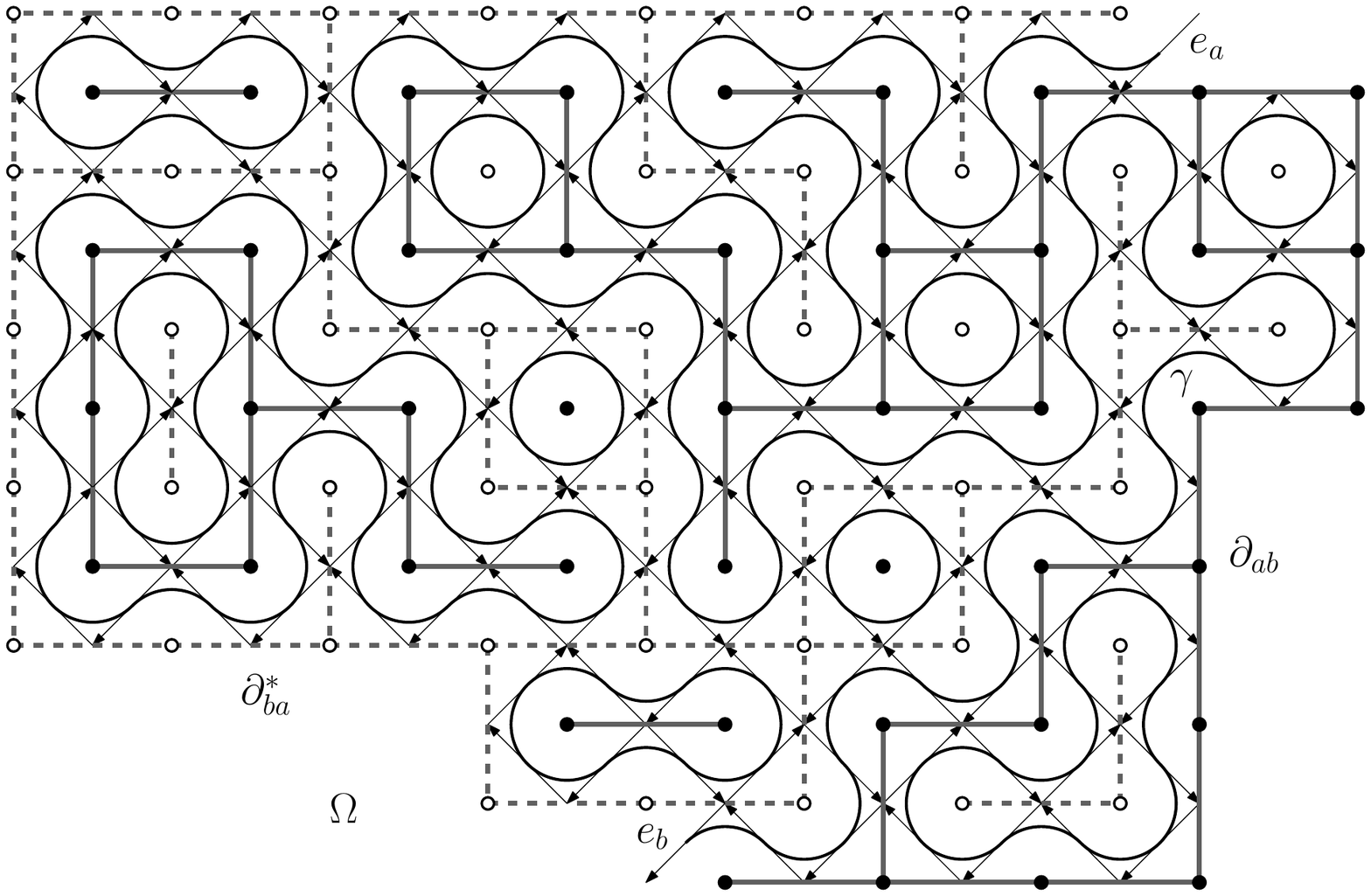}\caption{\label{fig:1}{\bf Left (top).} The configuration $\omega$ with its dual configuration $\omega^*$. \label{fig:2}{\bf Left (bottom).} The loop representation associated to $\omega$. 
\label{fig:3} {\bf Right.} A loop representation in a Dobrushin domain.}
\end{center}
\end{figure}

\paragraph{Loop representation} A third representation as a gas of loops has the advantage of attributing symmetric roles to the primal and the dual configuration. This representation corresponds to a fully packed $O(n)$-model on the square lattice. 

More precisely, consider a 
configuration $\omega$. It defines clusters in $G$ and dual clusters in 
$G^*$. Through every vertex of the medial graph $G^\diamond$ of $G$ passes either an open 
bond of $G$ or a dual open bond of $G^*$. For this reason, there is a unique way to 
draw an Eulerian (\emph{i.e.} using every edge exactly once) collection of loops on the 
medial lattice. Namely, a loop arriving at a vertex of the medial lattice 
always makes a $\pm \pi/2$ turn so as not to cross the open or dual open 
bond through this vertex, see Fig.~\ref{fig:2}. This gives a bijection between FK configurations on $G$ and 
Eulerian loop configurations on $G^{\diamond}$. 

The loops correspond to the \emph{interfaces} separating clusters from dual 
clusters. The probability measure 
can be nicely rewritten (using Euler's formula) in terms of the loop 
picture: for any configuration $\omega$,
\begin{eqnarray}
\phi_{G,p,q}^{a,b}(\omega)~=~\frac{1}{Z}~x^{o(\omega)}\sqrt q^{\ell(\omega)}
\end{eqnarray}
where $x=p/[\sqrt q (1-p)]$, $\ell(\omega)$ is the number of loops in the loop configuration associated to $\omega$, $o(\omega)$ is the number of open edges, and $Z$ is the normalization constant.

When considering Dobrushin boundary conditions on $(G,a,b)$, we obtain a slightly different representation, see Fig.~\ref{fig:3}. Besides 
loops, the configuration on $G^\diamond$ contains a single curve joining the edges 
 $e_a$ and $e_b$ between the arcs $\partial_{ba}$ and $\partial_{ab}^*$ (this is the dual arc adjacent to $\partial_{ab}$). This curve is called the 
\emph{exploration path} and is denoted by $\gamma$. It corresponds 
to the interface between the cluster connected to the wired arc $\partial_{ba}$ and the 
dual cluster connected to the free arc $\partial_{ab}^*$.

\paragraph{Definition of the observable} Fix a Dobrushin domain $(G,a,b)$. Following~\cite{Smi10}, an observable $F$ is now defined on the edges 
of the medial graph. Roughly speaking, $F$ is a modification of the 
probability that the exploration path passes through an edge. First, 
introduce the following definition. The \emph{winding} 
$\text{W}_{\Gamma}(z,z')$ of a curve $\Gamma$ between two edges $z$ and 
$z'$ of the medial graph is the total (signed) rotation (in radians) that the 
curve makes from the mid-point of the edge $z$ to that of the edge 
$z'$ (see Fig.~\ref{fig:winding}). 

\begin{definition}[Smi10,CR06]
Consider a Dobrushin domain $(G,a,b)$ and $0<q<4$. Define the {\em (edge) parafermionic observable} $F$ for any medial edge $e$ by
\begin{equation*}
  F(e) ~:=~ \phi_{G,p_c,q}^{a,b} \left({\rm e}^{{\rm i}\sigma 
  \text{W}_{\gamma}(e,e_b)} 1_{e\in \gamma}\right),
\end{equation*}
where $\gamma$ is the exploration path and $\sigma$ is given by the 
relation
$\displaystyle  \sin (\sigma \pi/2) = \sqrt{q}/2.$ 
\end{definition}
A (vertex) parafermionic observable can be defined on medial vertices by the formula $F(v):=\frac12\sum_{e\sim v}F(e)$
where the summation is over medial edges incident to $v$.
For $q\in[0,4]$, the observable $F$ is a  
 parafermion of spin $\sigma$, which is a real 
  number in $[0,1]$. 
%
%\begin{remark}
%Let us mention that the observable makes sense in the $q>4$ case. The spin $\sigma$ is not real anymore and does not have any physical interpretation. A natural question would be to relate this change of behavior for $\sigma$ to the transition between conformally invariant critical behavior and first order critical behavior. Note that the observable was used in \cite{BDCS11} to compute the critical surface of random-cluster model with cluster-weight $q>4$ on isoradial graphs.
%\end{remark}
%
\section{Properties of the parafermionic observable}\label{properties}
 
 The parafermionic observable possesses three fundamental properties that we now present. The first one is a local relation satisfied by the observable.

\begin{proposition}[local relation]
  \label{relation_around_a_vertex}
  Consider a medial vertex $v$ in $G^\diamond$ with four incident medial edges indexed $NW$, $SE$, $NE$ and $SW$ in the obvious way. Then,  \begin{equation}
    \label{rel_vertex}
    F(NW)-F(SE) ~=~i[F(NE)-F(SW)].
  \end{equation}
\end{proposition}

Since the proof is short and beautiful, we provide it here. The proof for the $q=2$ case is due to Smirnov \cite{Smi10}. The proof in the general case is a straightforward extension of Smirnov's lemma (it can also be found in various other places, including \cite{CR06}).

\begin{proof}
Let us assume that $v$ corresponds to a primal edge pointing $N$ to $S$. The case $E$ to $W$ is similar.
  
  We consider the involution 
  $s$ (on the space of configurations) which switches the state (open or 
  closed) of the edge of the primal lattice corresponding to $v$. 
  Let $e$ be an edge of the medial graph and denote by 
  \begin{eqnarray*}e_{\omega} &:=& 
  \phi_{G,p_c,q}^{a,b}(\omega) \, {\rm e}^{{\rm 
  i}\sigma {\rm W}_{\gamma(\omega)}(e,e_b)} 1_{e\in \gamma(\omega)}
  \end{eqnarray*}
  the contribution 
  of the configuration $\omega$ to $F(e)$.  Since $s$ is an involution, the following 
  relation holds: $$F(e)~=~\sum_{\omega} e_{\omega}~=~\frac{1}{2} 
  \sum_{\omega} \left[ e_{\omega}+e_{s(\omega)} \right]\!.$$ In order to 
  prove \eqref{rel_vertex}, it suffices to prove the following for any 
  configuration $\omega$:
  \begin{equation}
    \label{cd}
    NW_{\omega} + NW_{s(\omega)} - SE_{\omega} - SE_{s(\omega)} ~=~ i[
   NE_{\omega} + NE_{s(\omega)} - 
    SW_{\omega} -
    SW_{s(\omega)}].
  \end{equation}
  
  \begin{figure}[ht]
    \begin{center}
      \includegraphics[width=10cm]{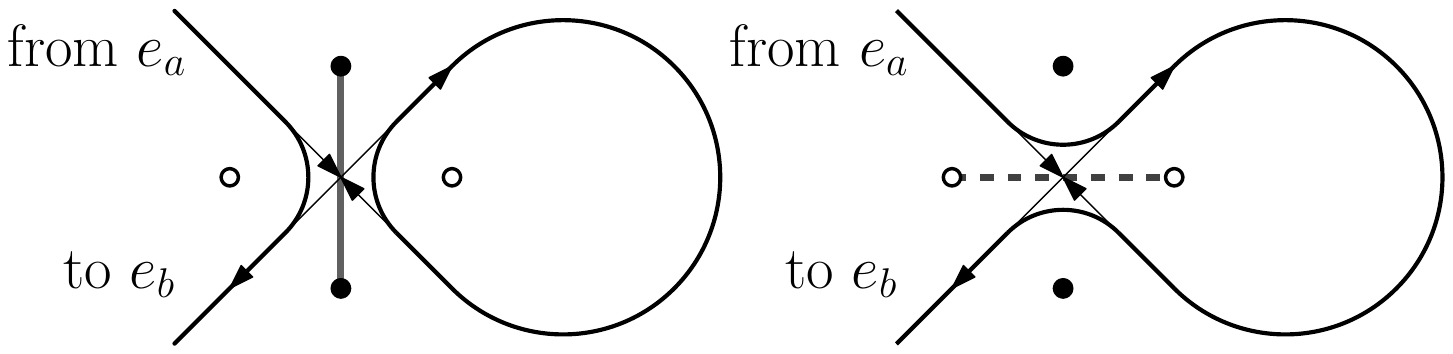}\quad\quad\quad\quad\quad\includegraphics[width=3cm]{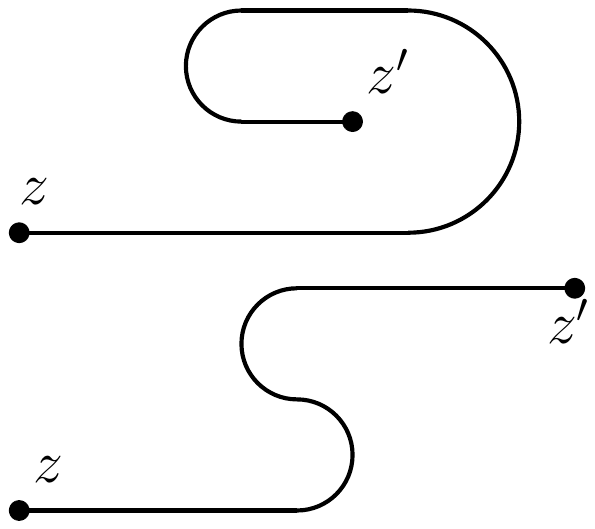}
    \end{center}
    \caption{    \label{fig:configuration} {\bf Left.} Two associated configurations $\omega$ and $s(\omega)$. \label{fig:winding}{\bf Right.} Two examples of winding. On the top, the winding is $-2\pi$, while on the bottom it is $0$.}

  \end{figure}
  
  \noindent There are three possibilities:
 \begin{itemize}
 \item[C1] $\gamma(\omega)$ does not go through any edge incident   
  to $v$. Then, neither does $\gamma(s(\omega))$. For any $e$ incident to $v$, we deduce that $e_{\omega}$ and $e_{s(\omega)}$ vanish and \eqref{cd} trivially holds.  
  
\item[C2] $\gamma(\omega)$ goes through two edges incident to $v$, see Fig.~\ref{fig:configuration}. Since $\gamma$ and the medial lattice are naturally oriented, $v$ enters through either $NW$ or $SE$ and leaves 
  through $NE$ or $SW$. Assume that $\gamma(\omega)$ 
  enters through the edge $NW$ and leaves through the edge 
  $SW$ (\emph{i.e.} that the primal edge corresponding to $v$ is open). The other cases are treated similarly. It is then
  possible to compute the contributions for $\omega$ and $s(\omega)$ of all the edges adjacent to $v$ in terms of $NW_{\omega}$. Indeed,
  \begin{itemize}
    \item The probability of $s(\omega)$ is equal to $1/\sqrt{q}$ 
      times the probability of $\omega$ (due to the fact that there is 
      one less open edge and one less loop of weight $\sqrt q$).
    \item Windings of the curve can be expressed using the winding at 
      $NW$. For instance, the winding of $NE$ in the configuration 
      $s(\omega)$ is equal to the winding of $NW$ minus $\pi/2$.
  \end{itemize}
  The contributions are given in the following table.
  \begin{center}\begin{tabular}{|c|c|c|c|c|}
    \hline
    configuration  &  $NW$ &  $SE$ &  $NE$ & $SW$\\
     \hline
    $\omega$ & $NW_{\omega}$ & 0 & 0 & ${\rm e}^{{\rm 
    i}\sigma\pi/2}NW_{\omega}$ \\
    \hline
    $s(\omega)$ & $NW_{\omega}/\sqrt q$ & ${\rm e}^{{\rm i}\sigma\pi}NW_{\omega}/\sqrt q$ & ${\rm 
    e}^{-{\rm i}\sigma\pi/2}NW_{\omega}/\sqrt q$ & ${\rm e}^{{\rm i}\sigma\pi/2}NW_{\omega}/\sqrt q$\\
    \hline
  \end{tabular}\end{center}
  Using the identity ${\rm e}^{{\rm i}\sigma\pi/2}-{\rm e}^{-{\rm 
  i}\sigma\pi/2}=i\sqrt{q}$, we deduce \eqref{cd} by summing (with the right weight) the contributions 
  of all the edges incident to $v$.
  
\item[C3] $\gamma(\omega)$ goes through the four edges incident to $v$. Then the exploration path of $s(\omega)$ goes through only two edges, and the computation is the same as in the second case.
  \end{itemize}
 In conclusion, \eqref{cd} is always satisfied and the claim is proved. 
 \end{proof}
%When $p=p_c$, $\alpha(p,q)=0$ and the previous relation becomes a discretization of the Cauchy-Riemann equation. Note that \eqref{rel_vertex} immediately translates into the relation:
%\begin{equation}\label{rel_vertex_tilde}\tilde F(N)+\tilde F(S)~=~[\tilde F(E)+\tilde F(W)].
%\end{equation}
%Even though this relation has no natural interpretation in terms of discrete complex analysis, it would sometimes be more convenient to handle than \eqref{rel_vertex}.
The relations \eqref{rel_vertex} can be understood as Cauchy-Riemann equations around vertices of $G$.  It implies that the integral of $F$ on any discrete contour vanishes. Interestingly, we do not know anything around vertices of $G^*$. Therefore, the observable is not preholomorphic according to the standard definition (see e.g. \cite{Smi10b}). Nevertheless, the sequence of parafermionic observables (on approximations $\delta \mathbb Z^2\cap \Omega$ of a given domain $\Omega$) is expected to converge uniformly (as $\delta\rightarrow0$) on any compact subset to a continuous function with vanishing integrals along closed contours. In such case, Morera's theorem implies that the limit is holomorphic. In order to identify the limit, it is therefore important to study the boundary conditions of the observable. 

\begin{proposition}[Boundary conditions]
  \label{boundary}
  Let $x\in G$ be a site on the free arc $\partial_{ab}$, and $e\in \partial G^\diamond$ be a medial edge adjacent to $x$. Then,
    \begin{equation*}
    F(e)~=~{\rm e}^{{\rm i}\sigma W(e,e_b)}\ \phi_{G,p_c,q}^{a,b}(x\longleftrightarrow \text{wired arc }\partial_{ba}),
  \end{equation*}
where $W(e,e_b)$ is the winding of an arbitrary curve on the medial lattice from $e$ to $e_b$.
\end{proposition}

%\begin{figure}[ht]
%  \begin{center}
%    \includegraphics[width=0.8\hsize]{boundary}
%  \end{center}
%  \caption{\textbf{Left:} A schematic picture of the exploration path
%    and a boundary point $u$, together with two possible choices $e_1$
%    and $e_2$ for $e$. If $u$ is connected to the wired arc, the
%    exploration path must go through $e$. \textbf{Right:} The winding of
%    a curve. In the first example, the curve did one quarter-turn on the
%    left and one quarter-turn on the right.}
%  \label{fig:boundary}
%\end{figure}

\begin{proof}
  Let $x$ be a site of the free arc $\partial_{ab}$ and recall that the exploration path
  is the interface between the open cluster connected to the wired arc $\partial_{ba}$
  and the dual open cluster connected to the free arc $\partial^*_{ab}$. Since $x$ belongs
  to the free arc, $x$ is connected to the wired arc if and only if $e$
  is on the exploration path. Therefore, $$\phi_{G,p,q}^{a,b}(x\longleftrightarrow
  \text{wired arc }\partial_{ba})=\phi_{G,p,q}^{a,b}(e\in \gamma).$$ The edge $e$ being on
  the boundary, the exploration path cannot wind around it, so that the
  winding of the curve is deterministic. Call it $W(e,e_b)$. We deduce from this
  remark that
  \begin{align*}
    F(e)~=~\phi_{G,p_c,q}^{a,b} ({\rm e}^{{\rm i}\sigma W(e,e_b)}
    1_{e\in \gamma}) &= ~{\rm e}^{{\rm i}\sigma W(e,e_b)}~\phi_{G,p_c,q}^{a,b}(e\in \gamma) \\ &=~{\rm e}^{{\rm i}\sigma W(e,e_b)}~\phi_{G,p_c,q}^{a,b}(x\longleftrightarrow \text{wired arc }\partial_{ba}).
  \end{align*}
\end{proof}

 The relation for dual sites near the wired arc can be deduced by duality (in such case the corresponding quantity is the $\phi_{G,p_c,q}^{a,b}$-probability that $v$ is connected by a dual open path to the free arc).
 
The previous proposition has two important consequences. The first one is the fact that the complex argument of the observable on the boundary is determined. At the discrete level, this corresponds to the fact that the observable is parallel to the normal vector to the power $-\sigma$ on the boundary. The second is the fact that the complex modulus of the observable equals the probability that a site on the boundary is connected to the wired arc by an open path. It enables us to relate probabilities of connections on the boundary to the observable.
\medbreak
Holomorphicity and the previous boundary conditions naturally identify the limit as the unique solution of a Riemann-Hilbert problem, and we obtain the following prediction: 
\begin{conjecture}[Smirnov]\label{FK parafermion}
Let $0<q< 4$ and $(\Omega,a,b)$ be a simply connected domain with two points on its boundary. For every $z\in \Omega$,
\begin{eqnarray}
\frac 1{(2\delta)^\sigma}F_{\delta}(z)~\rightarrow~\phi'(z)^\sigma\quad\text{when }\delta\rightarrow 0
\end{eqnarray}
where $\sigma=1-\frac 2\pi\arccos (\sqrt q/2)$, $F_\delta$ is the observable at $p_c$ in $G=\delta \mathbb Z^2\cap \Omega$ with spin $\sigma$, and $\phi$ is any conformal map from $\Omega$ to $\mathbb R\times(0,1)$ sending $a$ to $-\infty$ and $b$ to $\infty$.
\end{conjecture}

% In particular, proving that the observable is precompact seems an Herculean task with today's knowledge. 
%Nevertheless, it still provides us with a conditional result: if $\frac1{(2\delta)^{\sigma}}F_{\delta}$ converges, then the limit is $\phi'(z)^\sigma$. 

Importantly, $F$ is not determined by the collection of relations \eqref{rel_vertex} for general $q$ (the number of variables exceeds the number of equations) and a proof of this conjecture is still lacking. Let us mention a very important exception. For $q=2$, which corresponds to $\sigma=1/2$, the complex argument modulo $\pi$ of the edge-observable inside the domain depends on the orientation of the edge only (the winding takes value in $\theta+2\pi i \mathbb Z$ and therefore ${\rm e}^{{\rm i}\frac12 W_\gamma(e,e_b)}$ equals ${\rm e}^{{\rm i}\theta/2}$ or $-{\rm e}^{{\rm i}\theta/2}$). This specificity implies a stronger notion of discrete holomophicity for the observable, called $s$-holomorphicity; see \cite{Smi10,Smi10b,DCS11}. In particular, in such case the previous conjecture is a theorem due to Smirnov: the observable converges in the scaling limit to $\sqrt{\phi'(z)}$; see \cite{Smi10} again.
\medbreak
An interesting by-product of the conformal covariance of an observable is the following application. 
In all the known cases of convergence of discrete interfaces to $\mathsf{SLE}$, one starts a with conformally covariant observable of the system. After proving precompactness of interfaces in a relevant space of random Loewner chains (see \cite{Law05,KN04} for definitions), the so-called driving process of the Loewner chain can be identified using the conformal covariance of the observable together with L\'evy's theorem.  We refer to \cite{KS10,LSW04a,SS05} for examples of this scheme in the case of  Loop-Erased Random Walks, Uniform Spanning Trees, Harmonic Explorer, Ising model and to \cite{KS10} for a conditional result in the general case. In \cite{Smi06}, Smirnov proposed the following general program in order to prove convergence to $\mathsf{SLE}$.

\begin{enumerate}
\item Prove compactness of the interfaces. 
\item Show that sub-sequential limits are Loewner chains (with unknown random
driving process $W_t$). 
\item Prove the convergence of discrete observables of the model.
\item Extract from the limit of these observables enough information to evaluate
the conditional expectation and quadratic variation of increments of $W_t$. In order to do so, the observable in Step 3 will be chosen to be conformally covariant in the scaling limit, and to be a martingale for the interfaces. In such case, L\'evy's theorem and a small computation allows to identify $W_t$ to be $\sqrt{\kappa}B_t$, where $B_t$ is the standard Brownian motion. As a consequence, all sub-sequential limits must be $\mathsf{SLE}$($\kappa$). 
\end{enumerate}
For FK models, the first step has been proved in \cite{Dum11}. The second step is open for general $q\in(0,4)$, but is known for $q=1$ or 2. The third step should be the most difficult, and it has been implemented only for $q=2$ and 0. 
The choice of the observable in Step 3 is not determined uniquely. The main requirements to be able to implement Step 4 later on is that the observable is conformally covariant in the scaling limit and is a martingale of the discrete exploration path (and therefore its scaling limit must be a martingale for the limiting curve).
The simplest $\mathsf{SLE}$ martingales are given by  $g_t'(z)^\alpha [g_t(z)-W_t]^\beta$, where $\kappa=4(\alpha-\beta)/[\beta(\beta-1)]$. 
The (conjectured) limits of parafermionic observables are of the previous forms with $(\alpha,\beta)=(\sigma,-\sigma)$. Parafermionic observables can therefore be viewed as discretizations of very simple $\mathsf{SLE}$ martingales. 
In fact, discrete parafermionic observables are already martingales at the discrete level (with respect to the discrete exploration path). %More precisely, if $\gamma$ is indexed from $e_a$ to $e_b$.
\begin{proposition}[Martingale property]\label{Martingale}
Fix a Dobrushin domain $(\Omega,a,b)$. The FK fermionic observable $M_n(z)=F_{\Omega\setminus\gamma[0,n],\gamma_n,b}(z)$ is a martingale with respect to $(\mathcal F_n)$ where $\mathcal F_n$ is the $\sigma$-algebra generated by the FK interface $\gamma[0,n]$ (here the curve is parametrized by the number of lattice steps). 
\end{proposition}

\begin{proof}
For a Dobrushin domain $(\Omega,a,b)$, the slit domain created by ``removing'' the first $n$ steps of the exploration path is again a Dobrushin domain. Conditionally on $\gamma[0,n]$, the law of the FK percolation in this new domain is exactly $\phi_{\Omega^\diamond\setminus \gamma[0,n]}^{\gamma_n,b}$. Note that this is due to the Domain Markov property. This observation implies that $M_n(z)$ is the random variable $1_{z\in \gamma}{\rm e}^{{\rm i}\sigma W_{\gamma}(z,e_b)}$ conditionally to $\mathcal F_n$. Thus, it is automatically a martingale.
\end{proof}

In conclusion, the parafermionic observables provide us with a natural family of martingales for discrete exploration paths for which we know what the scaling limit should be. Therefore, the third step, which could a priori be performed with any well-chosen observable, can be done with the parafermionic observable and Step 3 boils down to Conjecture~\ref{FK parafermion}. 

Fix $q\in[0,4]$. Assuming that the three first steps have been implemented with the parafermionic observable, the fourth step of the program is easy. Conjecture~\ref{FK parafermion} implies that the observable is of the form $g_t'(z)^\alpha [g_t(z)-W_t]^\beta$ in the scaling limit, where $(\alpha,\beta)=(\sigma,-\sigma)$. In particular, it is a martingale for $\mathsf{SLE}(8/(\sigma+1))$. The last step will thus lead to the fact that the limit of discrete interfaces is $\mathsf{SLE}(8/(\sigma+1))$. By replacing $\sigma$ by its expression in terms of $q$, we obtain the following prediction

\begin{conjecture}[Schramm, \cite{Sch07}]\label{FK SLE}
The law of critical FK interfaces with cluster-weight $q\in [0,4]$ converges to the Schramm-Loewner Evolution with parameter $\kappa=\frac{4\pi}{\arccos(-\sqrt q/2)}$.\end{conjecture}

The previous discussion shows that conformal invariance in the scaling limit is not a required assumption  to obtain this conjecture. We only required that the parafermionic observable admits a scaling limit. Of course, this assumption is extremely hard to justify rigorously in general. 
\medbreak
The conjecture was proved by Lawler, Schramm and Werner \cite{LSW04a} for $q=0$: they showed that the perimeter curve of the uniform
spanning tree converges to $\mathsf{SLE}$(8). Note that
the loop representation with Dobrushin boundary conditions still makes sense for $q=0$ (more precisely for the model obtained by letting $q\rightarrow 0$ and $p/q\rightarrow 0$). In fact, configurations have no loops, just a curve running from $a$ to
$b$ (which then necessarily passes through all the edges), with all configurations being
equally probable. The $q=2$ case was proved in \cite{Smi10} and follows from the convergence of the parafermionic observable. In these cases, the spin is related to the central charge of the conformal field theory describing the critical behavior. This relation is expected to hold whenever $q\le 4$. For values of $q\in[0,4]\setminus\{0,2\}$, Conjecture~\ref{FK parafermion} and a fortiori Conjecture~\ref{FK SLE} are open. The $q=1$ case is particularly interesting, since it corresponds to bond percolation on the square lattice.

\section{Application to the study of the order of the phase transition}\label{sec:1<q<4}

Let us divide the proof of Theorem~\ref{thm:correlation length} into three cases. First, the easy case $1\le q\le 3$. Second, the slightly more technical case $3<q<4$. Third, the $q=4$ case, for which we introduce an alternative parafermionic observable. 
For $1\le q\le 3$, Theorem~\ref{thm:correlation length} will be a direct consequence of Theorem~\ref{thm:susceptibility}. For $3\le q\le 4$, we will in fact prove the following weak version of Theorem~\ref{thm:susceptibility}, which is also sufficient to imply Theorem~\ref{thm:correlation length}:
 \begin{proposition}\label{prop:poldecay}
 Let $q\in[1,4)$. There exists $\alpha=\alpha(q)>0$ such that 
 $$\phi^0_{\mathbb Z^2,p_c,q}(0\longleftrightarrow x)~\ge~\frac{1}{|x|^\alpha}.$$
 \end{proposition}

Before proving Theorem~\ref{thm:susceptibility} and Proposition~\ref{prop:poldecay}, let us show how it implies Theorem~\ref{thm:correlation length}.
 \begin{proof}[Theorem~\ref{thm:correlation length}]
For every $n,m>0$, the FKG inequality \eqref{FKG} implies, 
\begin{align*}\phi^0_{\mathbb Z^2,p,q}((0,0)\longleftrightarrow (n+m,0))~&\ge~\phi^0_{\mathbb Z^2,p,q}((0,0)\longleftrightarrow (n,0)\text{ and }(n,0)\longleftrightarrow (n+m,0))\\
&\ge~\phi^0_{\mathbb Z^2,p,q}((0,0)\longleftrightarrow (n,0))\cdot\phi^0_{\mathbb Z^2,p,q}((0,0)\longleftrightarrow (m,0))\end{align*}
which implies (by supermultiplicativity) that
$$\phi_{\mathbb Z^2,p,q}^0((0,0)\longleftrightarrow (n,0))~\le ~{\rm e}^{-n/\xi(p)},$$
where $\xi(p)$ is the correlation length. If $\xi(p)$ does not converge to $\infty$ as $p\nearrow p_c$, it increases to $\xi=\sup_{p<p_c}\xi(p)>0$. We thus obtain
$$\phi_{\mathbb Z^2, p_c,q}^0((0,0)\longleftrightarrow (n,0))=\lim_{p\nearrow p_c}\phi^0_{\mathbb Z^2,p,q}((0,0)\longleftrightarrow (n,0))\le \lim_{p\nearrow p_c}{\rm e}^{-n/\xi(p)}={\rm e}^{-n/\xi}.$$ In particular, $\phi_{\mathbb Z^2, p_c,q}^0((0,0)\longleftrightarrow (n,0))$ converges exponentially fast to 0, which is in contradiction with the polynomial decay of correlations (see Theorem~\ref{thm:susceptibility} for $1\le q\le 3$ or Proposition~\ref{prop:poldecay} for $3<q\le 4$).
 \end{proof}

\subsection{Proof of Theorem~\ref{thm:susceptibility} in the case $1\le q\le 3$}

Let $S_n$ be the graph given by the vertex set $[-n,n]^2\setminus \{(k,0),k> 0\}$ and edges linking nearest neighbors. It corresponds to a slit subdomain of $[-n,n]^2$. Set $\partial_n=\partial S_n\setminus \partial [-n,n]^2$. 
\begin{proposition}\label{prop:boo}
Fix $0<q\le 3$. There exists $C>0$ such that for every $n$, \begin{equation}\label{eq:bo1}
\sum_{\partial S_n}\delta_x\,\phi_{S_n,p_c,q}^0(0\longleftrightarrow x)~=~1,
\end{equation}
where $|\delta_x|\le C$ for every $x\in \partial S_n$ and $\delta_x\le 0$ for any $x\in \partial_n$.
\end{proposition}

\begin{proof}
Consider the FK percolation on $S_n$ with free boundary conditions. This model can be thought of as a FK percolation in a Dobrushin domain, where the wired arc is reduced to $\{0\}$. In such case, the exploration path $\gamma$ is the loop passing around 0. This loop corresponds to the boundary of the cluster of the origin. The parafermionic observable is defined in this domain as usual.
The equality \eqref{eq:bo1} is then the translation of the fact that the integral along the discrete contour composed of boundary medial edges is equal to 0. The facts that $\delta_x<0$ and $|\delta_x|\le C$ follow from a direct computation, which is provided in Appendix~\ref{appendix8}.
\end{proof}

We are now in a position to prove Theorem~\ref{thm:susceptibility} when $1\le q\le 3$.
\begin{proof}[Theorem~\ref{thm:susceptibility}]
 Fix $1\le q\le3$ and $p=p_c$. 
Equation \eqref{eq:bo1} can be restated as
$$\sum_{x\in \partial S_n\setminus\partial_n}~\delta_x\,\phi^0_{S_n,p_c,q}(0\longleftrightarrow x)=1-\sum_{x\in \partial_n}\delta_x\,\phi^0_{S_n,p_c,q}(0\longleftrightarrow x)\ge 1$$
since $\delta_x\le 0$ on $\partial_n$. Therefore, 
$$\sum_{x\in \partial S_n\setminus\partial_n}\phi^0_{S_n,p_c,q}(0\longleftrightarrow x)\ge\sum_{x\in \partial S_n\setminus\partial_n}\frac{\delta_x}{C}\,\phi^0_{S_n,p_c,q}(0\longleftrightarrow x)~\ge~\frac1C$$
where $C$ is defined in Proposition~\ref{prop:boo}. We find
\begin{align*}\sum_{x\in \mathbb Z^2}\phi_{\mathbb Z^2,p_c,q}^0(0\longleftrightarrow x)~&\ge~\sum_{n>0}\sum_{x\in \partial S_n\setminus\partial_n}\phi_{\mathbb Z^2,p_c,q}^0(0\longleftrightarrow x)\nonumber\\
&\ge~\sum_{n>0}\sum_{x\in \partial S_n\setminus\partial_n}\phi_{S_n,p_c,q}^0(0\longleftrightarrow x)\nonumber\ge~\sum_{n>0}\frac1C~=~\infty.\end{align*}
In the second inequality, we used the comparison between boundary conditions \eqref{comparison}. We also used the fact that $\partial S_n\setminus \partial_n\subset\partial[-n,n]^2$.
\end{proof}

\subsection{Proof of Proposition~\ref{prop:poldecay} in the case $3<q<4$}

A crucial feature of the previous proof is that $\delta_x\le 0$ for $x\in \partial_n$. This property allows to show that the sum of connectivity probabilities on $\partial S_n\setminus \partial_n$ is bounded from below. On $S_n$, this property is only true for $\sigma\ge \frac13$, i.e. $q\le 3$. For values of $q$ between 3 and 4, one needs to consider an enlarged domain $U_n$ in which the previous property is somehow still true. This domain is not planar anymore: it is a graph on the universal cover of the plane minus a point. 
The graph $\mathbb U$ is defined as follows (see Fig.~\ref{fig:U}): the vertex set is given by $\Z^3$ and the edge set by
 \begin{itemize}
 \item $[(x_1,x_2,x_3),(x_1,x_2+1,x_3)]$ for every $x_1,x_2,x_3\in \Z$,
 \item $[(x_1,x_2,x_3),(x_1+1,x_2,x_3)]$ for every $x_1,x_2,x_3\in \Z$ such that $x_1\neq 0$ or such that $x_1=0$ and $x_2\ge0$,
 \item $[(0,x_2,x_3),(1,x_2,x_3+1)]$ for every $x_2<0$ and $x_3\in \Z$.
 \end{itemize}
This graph has the shape of a spiral staircase and can be seen as a graph on the universal cover of $\mathbb R^2\setminus \{(1/2,-1/2)\}$. Its medial graph is defined similarly to the planar cases and is denoted by $\mathbb U^\diamond$. We also set $U_{n}=\{(x_1,x_2,x_3)\in\mathbb U:|x_1|,|x_2|\le n\}$ for every $n\ge1$.

\begin{figure}[t]
\begin{center}
\includegraphics[width=14cm]{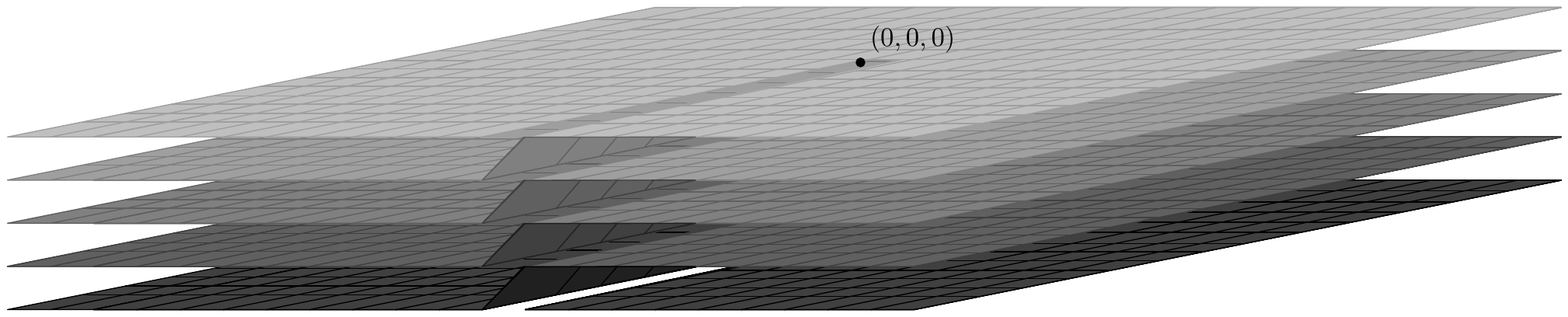}
\caption{\label{fig:U}The graph $\mathbb U$.}
\end{center}
\end{figure}

In this context, we obtain a proposition similar to Proposition~\ref{prop:boo}.

\begin{proposition}\label{prop:bo}
Fix $q< 4$. There exists $C>0$ such that for every $n$, \begin{equation}\label{eq:bo}
\sum_{\partial U_{n}}\delta_x\,\phi_{U_n,p_c,q}^0(0\longleftrightarrow x)~=~1,
\end{equation}
where $|\delta_x|\le C$ for every $x\in \partial U_n$.
\end{proposition}

\begin{proof}
The proof is roughly the same as  in Proposition~\ref{prop:boo}. The domain $U_n$ can be seen as an infinite Dobrushin domain, with wired arc $\{0\}$. In such case, $e_a$ and $e_b$ both correspond to the medial edge on $\partial U_n^\diamond$ adjacent to 0. By considering $e_a$ and $e_b$ as two different edges, the parafermionic observable can be defined similarly to the planar case. Furthermore, the local relations $$ F(NW)-F(SE) =i(F(NE)-F(SW))$$ is still valid since it only invokes the simple connectedness of $U_n$.

As before, \eqref{eq:bo} is then a consequence of the annulation of discrete contour integrals of this observable.
 
Note that the domain is infinite so that some additional care is required. Precisely, one needs to show that $F(e)$ and $\phi_{U_n,p_c,q}^0(0\longleftrightarrow x)$ go to 0 when $x$ and $e$ are taken far up or down compared to the origin. This comes from the following fact. If every edge of the form $[(y_1,0,y_3),(y_1+1,0,y_3)]$ for some fixed $y_3$ is closed, 0 cannot be connected to any $x$ with $x_3> y_3$. Since there are $n$ such edges, and that each one has a probability at least $1-p$ of being closed, we find that $\phi_{U_n,p_c,q}^0(0\longleftrightarrow x) \le [1-(1-p)^n]^{x_3}$. The same reasoning holds for the observable. 
\end{proof}

\medbreak
With the help of Proposition~\ref{prop:bo}, one can show Proposition~\ref{prop:poldecay}. The proof is slightly technical and we present it in Appendix~\ref{appendix:proof2}. The general philosophy is the same as in the previous section: integrating the observable on the boundary provides us with lower bounds on probability of being connected to the boundary of $U_n$ with free boundary conditions, which in turn imply that connectivity properties do not decay too fast. The additional difficulty comes from the fact that we originally work on $\mathbb U$ instead of $\mathbb Z^2$, and that we need to relate the behavior of FK percolation on $\mathbb U$ to its behavior on $\mathbb Z^2$. This is the reason for which we cannot prove  infinite susceptibility, but only polynomial decay of connectivity probabilities.

More generally, the relation between the behavior on $\mathbb U$ and $\mathbb Z^2$ is not clear, and a more systematic study should be performed.
%The proof goes along the following steps:
%\begin{enumerate}
%\item Use the same reasoning as in the proof of Theorem~\ref{thm:susceptibility} and invoke Proposition~\ref{prop:bo} to show that
%$$\sum_{x\in \partial U_{n}}\phi_{U_{n},p_c,q}^0(0\longleftrightarrow x)\ge \frac1C$$
%for $\theta=\theta(q)$ chosen in such a way that $\cos[(\sigma-1)(\theta+\frac\pi2)]\le 0$.
%
%\item Use these lower bounds to prove lower bounds on the probability for two points of a planar Dobrushin domain to be connected. Note that lower bounds in the first step are obtained in $\mathbb U$, meaning on a manifold which is not the plane. This step is therefore necessary to go from $\mathbb U$ to $\mathbb Z^2$. The prize to pay is to work in domains with Dobrushin boundary conditions instead of free boundary conditions.
%\item The last step consists in passing from Dobrushin boundary conditions to free boundary conditions.
%\end{enumerate}

\subsection{Proof of Proposition~\ref{prop:poldecay} in the case $q=4$}

When $q=4$, Smirnov's parafermionic observable becomes 
$F(e)={\rm e}^{{\rm i}W(e,e_b)}\phi_{G,p,4}^{a,b}(e\in \gamma)$.
Proposition~\ref{relation_around_a_vertex} then boils down to the fact that $\gamma$ enters and exists every vertex the same number of times. This fact is an easy implication of the fact that $\gamma$ is a curve and holds for every $p$. In particular, there is no hope for these relations to provide any insight on the phase transition. 

The reason for this loss of information is that we are not looking at the right observable. The observable becomes degenerated when $q\rightarrow 4$ and one should look at an expansion of the observable in powers of $(\sigma-1)$. Let us introduce
$$G(e)~:=~\phi_{G,p}^{a,b}[W_\gamma(e,e_b){\rm e}^{{\rm i}W_\gamma(e,e_b)}1_{e\in \gamma}].$$
\begin{proposition}
  \label{relation_around_a_vertex_4}
  Fix $q=4$ and $p=p_c(4)=2/3$. Consider a medial vertex $v$ in $G^\diamond$ with four incident medial edges, indexed in the obvious way. Then,  \begin{equation*}
    G(NW)-G(SE) ~=~i~[G(NE)-G(SW)].
  \end{equation*}
\end{proposition}
 
\begin{proof} Set $F_{q,\eta}(e)= \phi_{G,p_c,q}^{a,b} ({\rm e}^{{\rm i}\eta W(e,e_b)}
    1_{e\in \gamma}) $. Observe that for any $q<4$,
        \begin{eqnarray*}
    &F_{q,\sigma(q)}(NW)-F_{q,\sigma(q)}(SE)&=~i[F_{q,\sigma(q)}(NE)-F_{q,\sigma(q)}(SW)]\\
    & F_{q,1}(NW)-F_{q,1}(SE)&=~i[F_{q,1}(NE)-F_{q,1}(SW)],
    \end{eqnarray*}
    where $\sigma(q)$ satisfies $\sin(\sigma(q)\frac\pi 2)=\sqrt q/2.$
 Indeed, the first relation is due to Proposition~\ref{relation_around_a_vertex}, and the second follows readily from the fact that $\gamma$ is a curve (it simply asserts that a curve entering through $NW$ or $SE$ exits through $NE$ or $SW$). Now, since $\sigma(q)$ tends to $1$ as $q\nearrow 4$, we obtain the claim by using Rolle's lemma.
\end{proof}

 The observable $G$ plays the same role as the parafermionic observables for other values of $q$. In particular, it should converge in the scaling limit, when properly normalized, to 
$\log \phi'$
where $\phi$ is any conformal map from $\Omega$ to $\mathbb R\times(0,1)$ sending $a$ to $-\infty$ and $b$ to $\infty$. As before, the annulation of discrete contour integrals for $G$ allows us to implement the program introduced in the case $q<4$ to prove Theorem~\ref{thm:correlation length}. It starts by an analogue of Proposition~\ref{prop:bo}, which follows from the same proof.

 \begin{proposition}\label{cor:bo}
There exists $C>0$ such that for every $n$, \begin{equation*}
\sum_{\partial U_{n}}\delta_x\, \phi_{U_n,p_c,4}^0(0\longleftrightarrow x)~=~1,
\end{equation*}
where $|\delta_x|\le C$ for every $x\in \partial U_n$.
\end{proposition}

The proof of Proposition~\ref{prop:poldecay} is then identical to the case $q<4$. 
\section{Open questions}\label{open questions}

In conclusion, we discussed the existence of parafermionic observables in planar FK percolation on the square lattice. These observables have been introduced by Smirnov. Their integrals along discrete contours vanish, which enables us to 
\begin{enumerate}
\item[(a)] Predict the behavior in the scaling limit (this observation is due to Smirnov).

\item[(b)] Provide non-trivial information on the critical phase.
\end{enumerate}
Observables of the same type have been found in many other contexts. Furthermore, employing them to understand the model has been a fruitful strategy.
Let us conclude this article with some open questions (others than Conjectures 1 and 2).
\medbreak
\noindent 1. What information can be extracted from these observables in other models?
\medbreak
\noindent 2. How can we find systematically observables with vanishing discrete contour integrals? We already know parafermionic observables in the FK percolation, loop $O(n)$ models on hexagonal and square lattices, $Z_N$ models. They are one of the simplest examples of observables having this property, however they are not necessary the only one. 
\medbreak
\noindent 3. Exploring the relation between FK percolation on the plane or on the universal cover of the punctured plane is an interesting problem. It could appear to be useful when studying winding problems, in particular for the self-avoiding walk model.
\medbreak
\noindent 4. Probabilistic definitions of second order phase transitions are slightly different from Ehrenfest's one or the divergence of the correlation length. It usually involves uniqueness of infinite-volume measures with parameters $(p_c,q)$.
Even though different notions of the order of a phase transition are predicted to be the same, this equivalence has not been established in the case of FK percolation with general cluster weight. We therefore leave as an open problem to show that there is a unique FK percolation infinite-volume measure with parameter $(p_c,q)$, when $q\le4$.

 From classical arguments \cite[Theorem (5.33)]{Gri06}, it is sufficient to prove that there is no infinite cluster almost surely for the infinite-volume measure with wired boundary conditions denoted $\phi_{\mathbb Z^2,p_c,q}^1$. In the case of percolation, an argument of Russo \cite{Rus78} shows that the divergence of the susceptibility is equivalent to the absence of an infinite cluster in the dual. For $1<q\le3$, the mean-size of the cluster at the origin under $\phi^0_{\mathbb Z^2,p_c,q}$ was shown to be infinite in Theorem~\ref{thm:susceptibility}, which should be an indicator of the absence of a dual cluster. Since the dual model is a FK percolation at criticality with wired boundary conditions, the result would follow if Russo's argument could be extended to general FK percolations. Even though the argument seems fairly rigid, we were unable to generalize it.
 
Note that in the other direction, uniqueness of the infinite measure at criticality is sufficient to show that the transition is of second order. Indeed, it is classical that $\sum_{x\in \mathbb Z^2}\phi_{\mathbb Z^2,p_c,q}^1(0\leftrightarrow x)=\infty$, and the uniqueness implies directly that $\sum_{x\in \mathbb Z^2}\phi_{\mathbb Z^2,p_c,q}^0(0\leftrightarrow x)=\infty$.
\medbreak
\noindent 5. Parafermionic observables can be defined for $q>4$, see \cite{BDCS11}. In such case, the spin $\sigma$ is a complex number which is not real. It does not have any immediate physical relevance. Nevertheless, it is still possible to obtain relations comparable to \eqref{rel_vertex}. It would be interesting to relate the change of behavior of $\sigma$ to the change of behavior of the critical FK percolation (for $q>4$, it undergoes a first order phase transition).
\medbreak
\noindent 6. Using as an inspiration the works in \cite{BDCS11,Bax78}, it would also be interesting to extend our results to any isoradial graphs. The parafermionic observable is available there, and one should be able to make the proof work, with a substantial amount of new difficulties.

\appendix
\section{Appendix}

\subsection{Detailed derivation of Proposition~\ref{prop:boo}}\label{appendix8}

Fix $q< 3$, $p=p_c$ and drop them from the notation. Let $V=S_n^\diamond\setminus \partial S_n^\diamond$ be the set of medial vertices of $S_n^\diamond$ with four incident medial edges. For $v\in V$, the relation \eqref{rel_vertex} can be restated as
\begin{equation*}\sum_{e\text{ exiting } v}{\rm e}^{-{\rm i}W(e,e_b)}F(e)-\sum_{e\text{ entering } v}{\rm e}^{-{\rm i}W(e,e_b)}F(e)~=~0,\end{equation*}
where a medial edge incident to $v$ is entering $v$ if it is pointing toward $v$, and exiting otherwise. Summing the previous identity over all medial vertices in $V$, edges with two endpoints in $V$ disappear (since they are pointing towards one vertex of $V$, and outwards one of them). We obtain that
\begin{equation}\label{relll}\sum_{e\text{ exiting } V}{\rm e}^{-{\rm i}W(e,e_b)}F(e)-\sum_{e\text{ entering } V}{\rm e}^{-{\rm i}W(e,e_b)}F(e)~=~0,\end{equation}
where an edge enters $V$ if it is pointing toward a medial vertex of $V$ and away from a medial vertex of $V^c$, and it exits $V$ if it is pointing toward a medial vertex of $V^c$ and away from a medial vertex of $V$. 

Note that any edge entering or exiting $V$ is on the boundary.
Proposition~\ref{boundary} shows that for $e$ on the boundary, 
$$F(e)~=~{\rm e}^{{\rm i}\sigma W(e,e_b)}\phi_{S_n}^0(0\longleftrightarrow x)$$
where $x$ is the site bordered by $e$. Thus, \eqref{relll} implies
\begin{equation}\label{abcde}\sum_{x\in \partial S_n}\Big({\rm e}^{{\rm i}(\sigma-1) W(e_{\rm out}(x),e_b)}-{\rm e}^{{\rm i}(\sigma-1) W(e_{\rm in}(x),e_b)}\Big)\phi_{S_n}^0(0\longleftrightarrow x)=0,\end{equation}
where $e_{\rm in}(x)$ is the only medial edge bordering the face corresponding to $x$ and entering $V$, and $e_{\rm out}(x)$ is the only medial edge bordering the face corresponding to $x$ and exiting $V$.

Now, if $x=0$, we get that $e_{\rm in}(0)=e_a$ and $e_{\rm out}(0)=e_b$, and the associated  windings are $3\pi/2$ and $0$. The constant is therefore equal to $1-{\rm e}^{{\rm i}(\sigma-1)3\pi/2}=-2i\sin[(\sigma-1)\frac{3\pi}4]{\rm e}^{{\rm i}(\sigma-1)\frac{3\pi}4}$. By putting the contribution due to $x=0$ on the other side of the equal sign, and dividing by $2\sin[(\sigma-1)\frac{3\pi}4]{\rm e}^{{\rm i}(\sigma-1)\frac{3\pi}4}$, we find
\begin{equation}\label{abcdef}\sum_{x\in \partial S_n:x\ne 0}\Big(\frac{{\rm e}^{{\rm i}(\sigma-1) W(e_{\rm out}(x),e_b)}-{\rm e}^{{\rm i}(\sigma-1) W(e_{\rm in}(x),e_b)}}{2\sin[(\sigma-1)\frac{3\pi}4]{\rm e}^{{\rm i}(\sigma-1)\frac{3\pi}4}}\Big)\phi_{S_n}^0(0\longleftrightarrow x)=i.\end{equation}
Define for $x\ne 0$
\begin{align*}&\delta_x=\frac{1}{2\sin[(\sigma-1)\frac{3\pi}4]}\Im{\rm m}\Big({\rm e}^{{\rm i}(\sigma-1) (W(e_{\rm out}(x),e_b)-\frac{3\pi}4)}-{\rm e}^{{\rm i}(\sigma-1) (W(e_{\rm in}(x),e_b)-\frac{3\pi}4)}\Big)\\
&=\frac{\cos\Big[(\sigma-1)\big(\frac{W(e_{\rm out}(x),e_b)+W(e_{\rm in}(x),e_b)}{2}-\frac{3\pi}4\big)\Big]\sin\Big[(\sigma-1)(\frac{W(e_{\rm out}(x),e_b)-W(e_{\rm in}(x),e_b)}2)\Big]}{\sin[(\sigma-1)\frac{3\pi}4]}.\end{align*}
Obviously, $\delta_x$ has modulus smaller than $C:=1/\sin[(1-\sigma)\frac{3\pi}4)]<\infty$. Furthermore, if
$x\in \partial_n$, the entering edge has winding $2\pi$ (or $-\pi$ depending on which side of the slit it is) and the exiting edge has winding $5\pi/2$ (resp. $-\pi/2$). In both cases, the constant is equal to 
\begin{align*}\delta_x&=\frac{\cos[(\sigma-1)\frac{3\pi}2]\sin[(\sigma-1)\frac{\pi}{4}]}{\sin[(\sigma-1)\frac{3\pi}4]}.\end{align*}
This quantity is smaller than $0$ since $\frac{1}{3}\le \sigma< 1$.

\subsection{Proof of Proposition~\ref{prop:poldecay} in the case $1<q<4$}\label{appendix:proof2}

Fix $1<q<4$, $p=p_c$ and drop them from the notation.

The proof runs along the following lines. The main ingredient is once again \eqref{eq:bo}, which allows to show that there exists $x$ on the boundary of $U_n$ which is connected to the origin with good probability, even with free boundary conditions. The additional difficulty comes from the fact that we need to bootstrap this information to free boundary conditions on the plane. This part of the proof is technical and consists in playing with boundary conditions. We include it for completeness. The two first lemmas are not based on the observable and are valid for any $q\ge 1$.

\begin{lemma}\label{crossing}
For any $n\ge 1$, the probability that there exists a crossing from top to bottom in a square with wired boundary conditions on left and right, and free elsewhere, is larger than 1/2.
\end{lemma}

\begin{proof}
This is a simple consequence of self-duality. Observe that the complement of the event that there is an open path crossing from top to bottom in $[0,n]\times[0,n+1]$ is the event that there exists a dual open path from left to right in the dual graph. The dual measure is the measure with dual wired boundary conditions on left and right and free elsewhere on the dual graph, which is a rotated version of $[0,n]\times[0,n+1]$. Therefore, the probability of the complement event is the same as the probability of the event, i.e 1/2. We conclude by saying that the probability of having an open path crossing the square $[0,n]^2$ from top to bottom is larger than the one in $[0,n]\times[0,n+1]$.
\end{proof}

For $n,m\ge 1$, define $R(n,m)=[-n,n]\times[0,m]$. We also set $R_x(n,m)=x+R(n,m)$. For a rectangle $R$, let $
\partial_{*}R$ be the union of its top, left and right boundaries. Let $\phi^{\rm dobr}_{R(n,m)}$ be the FK measure on $R(n,m)$ with wired boundary condition on $\partial_*R(m,n)$ and free elsewhere. \begin{lemma}\label{estimate rectangle}
For any $n\ge 0$,
$$\phi^{\rm dobr}_{R(4n,n)}\Big(\textstyle (0,0)\longleftrightarrow \partial_{*} R(\frac n{16},\frac n4)\displaystyle\Big)\ge\frac{1}{16n^3}.$$
\end{lemma}

\begin{proof}
Consider the strip $\mathbb S_n=\mathbb Z\times[0,2n]$ of height $2n$. We fix wired boundary conditions on the top and free boundary conditions on the bottom.  Let $\mathcal E$ be the event that there exists an open path from $\{0\}\times[0,n]$ to the top of the strip. The complement of this event is contained in the event that there exists a dual open crossing from the dual arc $\{\frac12\}\times[n+\frac12,2n+\frac12]$ to the bottom of the strip. By symmetry and a self-duality argument similar to the proof of the previous lemma, we deduce that the probability of $\mathcal E$ is larger or equal to $1/2$.

Next, we claim that on $\mathcal E$, there exists $x\in R(4n^2,n)$ such that $x$ is connected to $\partial_*R_x(4n,n)$ and $x$ is not connected to $x+[-4n,4n]\times\{-1\}$, which we denote by $\partial_-R_x(4n,n)$ (this is the segment just below the bottom of $R_x(4n,n)$). We denote this event by $A(x)$. One can see that by looking at the lowest path from $\{0\}\times[0,n]$ to the top, and by studying its local minima. Therefore
\begin{align*}
\sum_{x\in R(4n^2,n)}\phi_{\mathbb S_n}^{\rm dobr}(A(x))&\ge \phi_{\mathbb S_n}^{\rm dobr}(\exists x\in R(4n^2,n):A(x))\ge \phi_{\mathbb S_n}^{\rm dobr}(\mathcal E)\ge \frac12,\end{align*}
where $\phi_{\mathbb S_n}^{\rm dobr}$ is the FK measure with wired boundary condition on the top, and free on the bottom.

Next, we aim to prove that $\phi_{\mathbb S_n}^{\rm dobr}(A(x))\le \phi^{\rm dobr}_{R(4n,n)}(\textstyle 0\longleftrightarrow \partial_{*} R(\frac{n}{16},\frac n4))$. This will imply the result immediately. 
We will be using another feature of the FK percolation, called the domain Markov property \cite[Lemma (4.13)]{Gri06}. In words, conditioned on the state of the edges outside of some graph $G$, the measure inside $G$ is a FK percolation with boundary conditions inherited from wiring induced by open edges outside $G$. This is the equivalent of the DLR property for Gibbs measures.

The event $A(x)$ is the intersection of the event that $x$ is connected to $\partial_*R_x(4n,n)$, and the event that it is not connected to $\partial_-R_x(4n,n)$. Conditioning on this second event boils down to conditioning on the lowest dual-open path, denoted $\Gamma$, disconnecting $x$ from $\partial_-R_x(4n,n)$ in $R_x(4n,n)$, see Fig.~\ref{fig:case0}. Conditionally on $\Gamma$, there must exist a path in the sites above it connecting $x$ to $\partial_*R_x(4n,n)$ in order for $A(x)$ to be verified. Let $S$ be the set of sites in $R_x(4n,n)$ above $\Gamma$. We deduce that
\begin{align*}&\phi_{\mathbb S_n}^{\rm dobr}(x\longleftrightarrow \partial_*R_x(4n,n)|\Gamma)=\phi_S^{\rm \xi}(x\longleftrightarrow \partial_*R_x(4n,n))\\
&\le \phi_{R_x(4n,n)}^{\rm dobr}(x\longleftrightarrow \partial_*R_x(4n,n)\text{ in S})\le\phi_{R_x(4n,n)}^{\rm dobr}(x\longleftrightarrow \partial_*R_x(4n,n)).
\end{align*}
Above, $\xi$ are the boundary conditions on $\partial S$ induced by the conditioning on $\Gamma$. In the first equality, we used the Domain Markov property. The first inequality is due to the following fact: since sites of $\Gamma$ are dual-open, the boundary conditions on $S$ are  dominated by those induced by free boundary conditions on the bottom of $R_x(4n,n)$ and wired on the three other sides of $R_x(4n,n)$. The last inequality is obvious. Note that the previous bound is uniform in the possible realizations of $\Gamma$. 
We find
\begin{align*}
\phi_{\mathbb S_n}^{\rm dobr}(A(x))&=\phi_{\mathbb S_n}^{\rm dobr}\Big(\phi_{\mathbb S_n}^{\rm dobr}(x\longleftrightarrow \partial_*R_x(4n,n)|\Gamma)1_{x\not\longleftrightarrow \partial_-R_x(4n,n)}\Big)\\
&\le \phi_{\mathbb S_n}^{\rm dobr}\Big(\phi^{\rm dobr}_{R(4n,n)}(\textstyle x\longleftrightarrow \partial_{*} R_x(4n,n)\displaystyle)1_{x\not\longleftrightarrow \partial_-R_x(4n,n)}\Big)\\
&\le \phi^{\rm dobr}_{R(4n,n)}\Big(\textstyle x\longleftrightarrow \partial_{*} R_x(4n,n)\displaystyle\Big)\le \phi^{\rm dobr}_{R(4n,n)}\Big(\textstyle (0,0)\longleftrightarrow \partial_{*} R(\frac{n}{16},\frac n4)\displaystyle\Big).\end{align*}
\end{proof}

\begin{lemma}\label{long crossing}
There exists $c_1>0$ such that 
$$\phi^{\rm dobr}_{R(n,n)}\Big(\textstyle R(n,\frac n4)\text{ contains an open path from left to right}\displaystyle\Big)~\ge~\frac{1}{n^{c_1}}.$$
\end{lemma}

\begin{proof}
Notice that if one shows that there exists $c>0$ such that
\begin{equation}\phi_{R(2n,n)}^{\rm dobr}\Big(\textstyle(0,0)\leftrightarrow (\frac{n}{8},0)\text{ in }R(2n,\frac n4)\Big)\ge \frac{1}{n^{c}},\label{eqr}
\end{equation}
then the comparison between boundary conditions implies that
$$\phi_{R(n,n)}^{\rm dobr}\Big(\textstyle x\leftrightarrow x+(\frac{n}{8},0)\text{ in }R(n,\frac n4)\Big)\ge \frac{1}{n^{c}}$$
for any $x$ on the bottom of $R(n,n)$ (note that the left, right and top sides are wired and count as connections). 

The FKG inequality \eqref{FKG} then implies that $x$ and $x+k\frac n8$ are connected with probability larger than $n^{-kc}$. Using this estimate for $x=(-n,0)$ and $k=16$ yields the result with $c_1=16c$.
In conclusion, we only need to show \eqref{eqr}.
Applying Lemma~\ref{estimate rectangle}, we face two cases, see Fig.~\ref{fig:case3}. 
\bigbreak
\noindent\textbf{Case 1:} $\displaystyle\phi^{\rm dobr}_{R(4n,n)}\Big(\textstyle(0,0)\longleftrightarrow \{\frac n{16}\}\times[0,\frac n4]\text{ in }R(\frac n{16},\frac n4)\displaystyle\Big)~\ge~\frac{1}{64n^3}.$
\bigbreak
\noindent In such case, there exists $x\in\{\frac n{16}\}\times[0,\frac n4]$ such that
$$\displaystyle\phi^{\rm dobr}_{R(4n,n)}(\textstyle(0,0)\longleftrightarrow x \text{ in }R(\frac n{16},\frac n4))~\ge~\frac{1}{16n^{4}}.$$
By symmetry and comparison between boundary conditions, we obtain that
\begin{align*}
&\phi^{\rm dobr}_{R(2n,n)}\Big(\textstyle(0,0)\longleftrightarrow x \text{ in }R(\frac n{16},\frac n4)\displaystyle\Big)\text{ and }\phi^{\rm dobr}_{R(2n,n)}\Big(\textstyle(\frac n{8},0)\longleftrightarrow x \text{ in }R_{(\frac n8,0)}(\frac n{16},\frac n4)\displaystyle\Big)
\end{align*}
are larger than $\frac{1}{16n^{4}}$. The FKG inequality implies \eqref{eqr} in this case.
\bigbreak
\noindent\textbf{Case 2:} $\displaystyle\phi^{\rm dobr}_{R(4n,n)}\Big(\textstyle (0,0)\longleftrightarrow [-\frac{n}{16},\frac{n}{16}]\times\{\frac n4\}\text{ in }R(\frac n{16},\frac n4)\displaystyle\Big)~\ge~\frac{1}{32n^3}.$
\bigbreak
Consider the event that there exists an open path in $R(\frac n{16},\frac n4)$ from 0 to the top, and an open path in $R_{(\frac n8,0)}(\frac n{16},\frac n4)$ from $(\frac n8,0)$ to the top. The FKG inequality implies that this event has probability larger than $1/(32n^3)^2$. 

We now aim to show that both vertical crossings can be connected by an open path with probability $1/2$. We use a technique close to the one used in the previous proof. 

Conditioning on the existence of the two previous open paths boils down to conditioning on the left most open path from $(0,0)$ to the top in $R(\frac n{16},\frac n4)$, and the right most open path from $(\frac n8,0)$ to the top. The part of $R(2n,\frac n4)$ in between these two paths is denoted $S$. Note that $S$ is included in $B=[-\frac n{16},\frac{3n}{16}]\times[0,\frac n4]$. As in the previous lemma, the Markov domain property and the comparison between boundary conditions imply that the boundary conditions on $\partial S$  dominate wired boundary conditions on the left and right of $B$, and free boundary conditions on the top and bottom. Using Lemma~\ref{crossing} and the same strategy as for $A(x)$, the probability that there exists an open path in $S$ from left to right (i.e. from the left most open path from $(0,0)$ to the top in $R(\frac n{16},\frac n4)$, to the right most open path from $(\frac n8,0)$ to the top in $R_{(\frac n8,0)}(\frac n{16},\frac n4)$) is larger than the probability that there exists an open path  crossing the square from left to right, i.e. $\frac12$. In particular, we created an open connection between $(0,0)$ and $(\frac n8,0)$. Overall, the probability that there exists an open path from $(0,0)$ to $(\frac n8,0)$ in $R(2n,\frac n4)$ is larger than $\frac12\frac{1}{(32 n^3)^2}\ge \frac{1}{n^c}$ for some $c$ large enough.\end{proof}

\begin{figure}[t]
\begin{center}
\includegraphics[width=10cm]{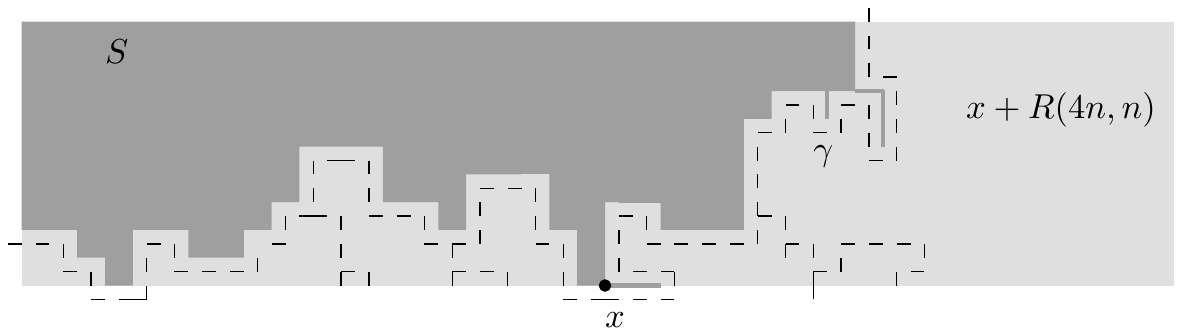}\quad\quad \includegraphics[width=4.8cm]{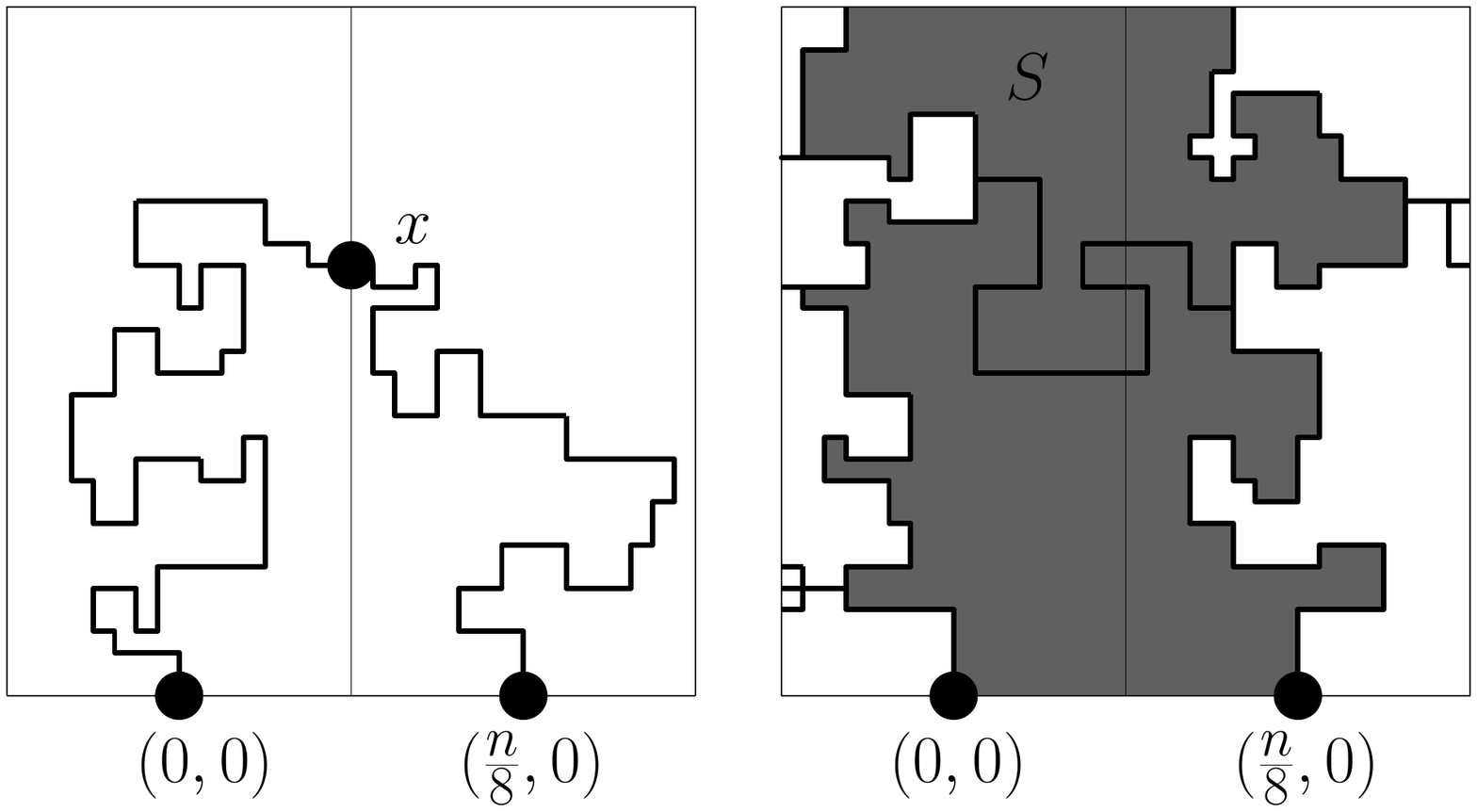}
\caption{\label{fig:case0}{\bf Left.} The path $\Gamma$ in $R_x(4n,n)$ and the area $S$ above it. \label{fig:case3}{\bf Right.} The two cases leading to a path from $(0,0)$ to $(\frac n8,0)$.}
\end{center}
\end{figure}

We now use \eqref{eq:bo} to deduce an estimate on crossing probabilities in a slit domain; see Fig.~\ref{fig:case1}. This is the only place where we use \eqref{eq:bo}. In particular, previous lemmas are true for $q>4$, where a first order phase transition is expected. The following lemma would not be valid for $q>4$.

Let $\partial_n=\{0\}\times[0,n]$ and let $C_n$ be the slit domain obtained by removing from $[-n,n]^2$ the  edges of $\partial_n$. Let $\phi^{\rm dobr}_{C_n}$ be the measure on $C_n$ with wired boundary conditions on $\partial_n$ and free elsewhere.
\begin{lemma}\label{lem:crucial}
There exists $c_2>0$ such that for any $n\ge 1$, 
$$\phi^{\rm dobr}_{C_n}\Big(\textstyle(0,-n)\longleftrightarrow \partial_*((0,-n)+R(\frac n{16},\frac n4))\Big)~\ge~ \displaystyle\frac{1}{n^{c_2}}.$$
\end{lemma}

\begin{proof}
In this proof, we are working on $\mathbb U$. For this reason, we use coordinates on $\mathbb Z^3$. 

The $\phi^0_{U_n}$-probability that the dual vertex $(-\frac12,-\frac12,k)$ is connected by a dual open path inside $[-n,0]\times[-n,n]\times\{k\}$ to  $\partial U_n$, conditionally to any configuration outside this rectangle, is larger than $\frac1{32n^3}$. Indeed, boundary conditions are free on every edge of the rectangle, except on the bottom one, for which it is  dominated by wired ones. Recall that free boundary conditions correspond to wired boundary conditions in the dual, and vice versa. Therefore, for the dual FK percolation on the dual graph of $[-n,0]\times[-n,n]\times\{k\}$, the boundary conditions above dominate wired boundary conditions on three sides, and free on the side containing $(-\frac12,-\frac12,k)$. Lemma~\ref{estimate rectangle} applied to the dual model implies the lower bound.

For a vertex $x=(x_1,x_2,x_3)\in \partial U_n$ (assume $x_3\ge 0$) to be connected to $(0,0,0)$, none of the dual vertices $(-\frac12,-\frac12,k)$ must be dual connected to $\partial U_n$ in $R_n\times\{k\}$, for $k\in [0,x_3]$. In particular, the previous lower bound implies
\begin{align*}\phi^0_{U_{n}}((0,0,0)\longleftrightarrow x)~&\le~ \Big(1-\frac1{32n^3}\Big)^{|x_3|}.\end{align*}
For $|x_3|\ge n^4$, $\phi^0_{U_n}((0,0,0)\longleftrightarrow x)$ becomes negligible. The previous equation together with \eqref{eq:bo} thus implies that there exist $c>0$ and $x\in \partial U_n$ with 
$$\phi^0_{U_{n}}((0,0,0)\longleftrightarrow x)\ge \frac{1}{n^{c}}.$$
Let us rotate and translate vertically $U_n$ in such a way that $x=(x_1,-n,0)$ for some $x_1\in[-n,n]$. Let us assume without loss of generality that $x_1\ge 0$. The boundary conditions for the primal model on $C_n$ induced by free boundary conditions on $U_n$ are  dominated by wired boundary conditions on $\partial_n$, and free elsewhere. We deduce that
\begin{equation}\label{eq:eq}\phi^{\rm dobr}_{C_n}(\textstyle x\longleftrightarrow \partial_n)~\ge~ \displaystyle\frac{1}{n^{c}}.\end{equation}
Now, we aim to bound from below the probability that $(0,-n,0)$ is connected to $\partial_n$ in $C_n$. Since we work on a planar domain, we drop the third coordinate from the notation. Assume that $x=(x_1,-n)$ and $(-x_1,-n)$ are connected to $\partial_n$. Consider the right-most open crossing from $(x_1,-n)$ to $\partial_n$, and the left-most open crossing from $(-x_1,-n)$ to $\partial_n$. Let $S$ be the component of $C_n$ between these two paths which contains $(0,-n)$, see Fig.~\ref{fig:case1}. The same strategy as for $A(x)$ implies that the boundary conditions in $S$  dominate free boundary conditions on the bottom of $C_n$, and wired elsewhere. Lemma~\ref{estimate rectangle} thus implies that
\begin{align*}&\phi^{\rm dobr}_{C_n}\Big(\textstyle(0,-n)\longleftrightarrow \partial_*R_{(0,-n)}(\frac n{16},\frac n4)\Big)\\
&\ge \frac{1}{n^{2c}}\phi^{\rm dobr}_{C_n}\Big(\textstyle(0,-n)\longleftrightarrow \partial_*R_{(0,-n)}(\frac n{16},\frac n4))~\Big|~(x_1,-n)\text{ and }(-x_1,-n)\longleftrightarrow \partial_n\Big)\\
&\ge \frac{1}{n^{2c}}\phi^{\rm dobr}_{R(n,2n)}\Big(\textstyle(0,0)\longleftrightarrow \partial_*R(\frac n{16},\frac n4)\Big)\ge \frac{1}{32\cdot(2n)^3n^{2c}}.\end{align*}
We used \eqref{eq:eq} in the first inequality, and in the last, the fact that the boundary conditions on $R(\frac n{16},\frac n4)$ conditioned on the event that $(x_1,-n)$ and $(-x_1,-n)$ are connected to $\partial_n$  dominate Dobrushin boundary conditions on $R(n,2n)$. The claim follows by choosing $c_2>0$ large enough.
\end{proof}

\begin{figure}[t]
\begin{center}
\includegraphics[width=7cm]{figure7.eps}\quad\quad\quad \includegraphics[width=7cm]{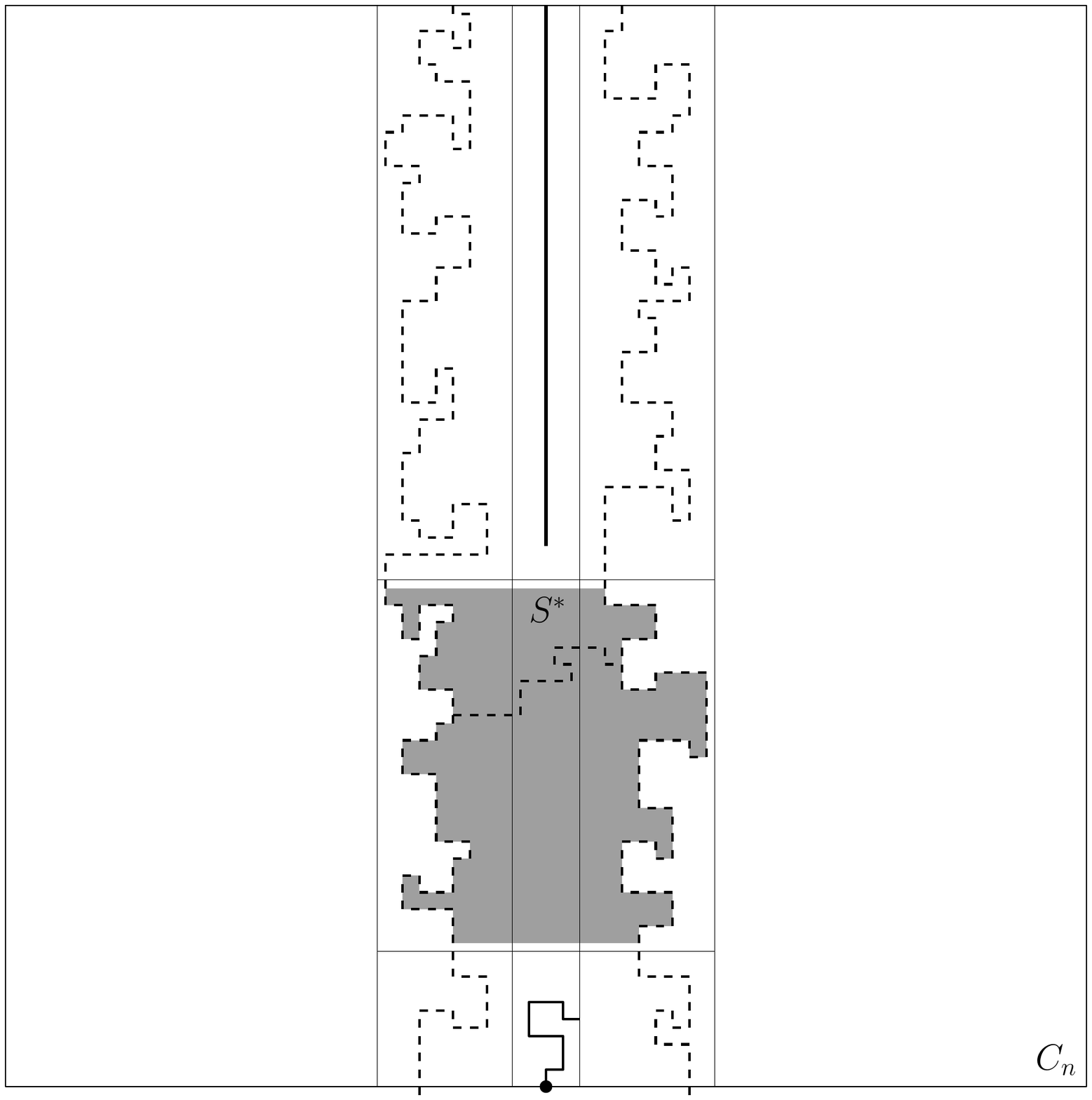}
\caption{\label{fig:case1}{\bf Left.} The two paths connecting $\partial_n$ to $(x_1,-n,0)$ and $(-x_1,-n,0)$ and the area $S$ between them. \label{fig:case2}{\bf Right.} The two dual-open paths in the long rectangles $[\frac n{16},\frac{5n}{16}]\times[-n,n]$ and $[-\frac{5n}{16},\frac n{16}]\times[-n,n]$.}
\end{center}
\end{figure}

\begin{proof}[Proposition~\ref{prop:poldecay}]
Define $\mathcal E=\{(0,-n)\longleftrightarrow \partial_*R_{(0,-n)}(\frac n{16},\frac n4)\}$ and $\mathcal F_{\rm right}$ and $\mathcal F_{\rm left}$ to be the events that rectangles $[\frac n{16},\frac{5n}{16}]\times[-n,n]$ and $[-\frac{5n}{16},-\frac{n}{16}]\times[-n,n]$ contain a dual open path from top to bottom. Let $\mathcal C$ be the event that there exists a dual open path in the square $[-\frac{5n}{16},\frac{5n}{16}]\times[-\frac{3n}{4},-\frac{n}{8}]$, connecting a dual open path crossing $[\frac n{16},\frac{5n}{16}]\times[-n,n]$ from top to bottom to a dual open path crossing $[-\frac{5n}{16},-\frac n{16}]\times[-n,n]$ from top to bottom.
First observe that
\begin{align*}\phi^0_{\mathbb Z^2}(\textstyle0\leftrightarrow \partial [-\frac{n}{16},\frac{n}{16}]^2)\ge \phi^{\rm dobr}_{C_n}(\mathcal E| \mathcal F_{\rm left}\cap \mathcal F_{\rm right}\cap \mathcal C)\ge \phi^{\rm dobr}_{C_n}(\mathcal E\cap \mathcal F_{\rm left}\cap \mathcal F_{\rm right}\cap \mathcal C)\end{align*}
Indeed, conditioned on $\mathcal F_{\rm left}\cap \mathcal F_{\rm right}\cap \mathcal C$, boundary conditions for the primal model on $R_{(0,-n)}(\frac n{16},\frac n4)$ are  dominated by free boundary conditions in the plane. It is therefore sufficient to prove a polynomial lower bound on the right-hand term. Trivially,
\begin{align*}\phi^{\rm dobr}_{C_n}(\mathcal E\cap \mathcal F_{\rm left}\cap \mathcal F_{\rm right}\cap \mathcal C)~=\phi^{\rm dobr}_{C_n}(\mathcal E)\cdot\phi^{\rm dobr}_{C_n}(\mathcal F_{\rm left}\cap \mathcal F_{\rm right}|\mathcal E)\cdot\phi^{\rm dobr}_{C_n}(\mathcal C|\mathcal E\cap\mathcal F_{\rm left}\cap \mathcal F_{\rm right})\end{align*}
Now, $\phi^{\rm dobr}_{C_n}(\mathcal E)\ge \frac1{n^{c_2}}$ by Lemma~\ref{lem:crucial}. Furthermore, conditioned on everything on the left of $\{\frac n{16}\}\times[-n,n]$, boundary conditions for the primal model on $[\frac n{16},n]\times[-n,n]$ are  dominated by wired boundary conditions on the left side and free elsewhere. In particular, boundary conditions for the dual model stochastically dominate free boundary conditions on the left side and wired elsewhere. It is thus possible to use Lemma~\ref{long crossing} in the dual model to bound from below the conditional probability of $\mathcal F_{\rm right}$ (existence of a vertical dual open crossing of $[\frac n{16},\frac{5n}{16}]\times[-n,n]$) by the quantity
$\frac{1}{n^{c_1}}$.

Similarly, conditioned on everything on the right of $\{-\frac n{16}\}\times[-n,n]$, boundary conditions for the primal model on $[-n,-\frac n{16}]\times[-n,n]$ are  dominated by wired boundary conditions on the right side and free elsewhere, and the same bound follows. 

Finally, we estimate $\phi^{\rm dobr}_{C_n}(\mathcal C|\mathcal E\cap\mathcal F_{\rm left}\cap \mathcal F_{\rm right})$. We focus on the configuration inside the square $[-\frac{5n}{16},\frac{5n}{16}]\times[-\frac{3n}{4},-\frac{n}{8}]$. As usual, condition on left and right most dual-open paths. Take $S^*$ to be the area of the dual graph in $[-\frac{5n}{16},\frac{5n}{16}]\times[-\frac{3n}{4},-\frac{n}{8}]$ between the right most dual open path from top to bottom in $[\frac n{16},\frac{5n}{16}]\times[-n,n]$, and the left most dual open path crossing from top to bottom in $[-\frac{5n}{16},\frac n{16}]\times[-n,n]$, see Fig.~\ref{fig:case2}. The boundary conditions for the dual model on $S^*$ dominate (dual) free boundary conditions on top and bottom, and (dual) wired elsewhere. The domain Markov property and the comparison between boundary conditions allow us to push (dual) wired boundary conditions to the left and right sides of $[-\frac{5n}{16},\frac{5n}{16}]\times[-\frac{3n}{4},-\frac{n}{8}]$, so that boundary conditions for the dual model on $S^*$ dominate (dual) free boundary conditions on top and bottom sides of $[-\frac{5n}{16},\frac{5n}{16}]\times[-\frac{3n}{4},-\frac{n}{8}]$, and (dual) wired on the two other sides. Therefore, the probability of having a dual open path in $S^*$ crossing from left to right is larger than $1/2$, thanks to Lemma~\ref{crossing}. In particular, 
$$\phi^{\rm dobr}_{C_n}(\mathcal C|\mathcal E\cap\mathcal F_{\rm left}\cap \mathcal F_{\rm right})\ge \frac12.$$
Putting everything together, we find that
$\phi^{\rm dobr}_{C_n}(\mathcal E\cap \mathcal F_{\rm left}\cap \mathcal F_{\rm right})\ge \frac{1}{2n^{c_2+2c_1}}$
and the claim follows.
\end{proof}

\bibliographystyle{plain}
\bibliography{bibli}

\begin{flushright}\footnotesize\obeylines  \textsc{Section de Math\'ematiques}
  \textsc{Universit\'e de Gen\`eve}
  \textsc{Gen\`eve, Switzerland}
  \textsc{E-mail:} \texttt{hugo.duminil@unige.ch}
\end{flushright}
\end{document}